\documentclass[a4paper,11pt]{amsart}
\usepackage{tikz}
\usepackage{amsfonts,amsmath,amssymb}
\usepackage{color}
\usepackage{algorithmic}
\usepackage{algorithm}
\usepackage{comment}
\usepackage{graphics}
\usepackage[hidelinks]{hyperref}

\usepackage{listings}
\usepackage{mathtools}
\usepackage{esint}
\newtheorem{theorem}{Theorem}

\newtheorem{corollary}[theorem]{Corollary}
\theoremstyle{definition}
\newtheorem{definition}[theorem]{Definition}

\theoremstyle{lemma}
\newtheorem{lemma}[theorem]{Lemma}

\theoremstyle{remark}
\newtheorem{remark}[theorem]{Remark}

\numberwithin{theorem}{section}
\numberwithin{equation}{section}
\numberwithin{table}{section}
\numberwithin{figure}{section}

\usepackage{caption}
\usepackage{subcaption}
\usepackage{graphics}

\definecolor{aliceblue}{rgb}{0.94, 0.97, 1.0}

\newcommand{\calT}{\mathcal{T}}
\newcommand{\R}{\mathbb{R}}
\newcommand{\N}{\mathbb{N}}

\newcommand{\opnorm}[1]{{\left\vert\kern-0.25ex\left\vert\kern-0.25ex\left\vert #1 \right\vert\kern-0.25ex\right\vert\kern-0.25ex\righ\vert}}

\newcommand\dx{\,\text{d}x}

\newcommand{\VH}{V_H}

\newcommand{\Vh}{V_h}
\newcommand{\V}{H^1_0(D)}

\newcommand{\LL}{L^2(D)}
\newcommand\calR{\mathcal{R}}

\newcommand\calQ{\mathcal{Q}}

\newcommand\calI{\mathcal{I}}

\newcommand{\eps}{\varepsilon}

\newcommand{\W}{\mathcal{W}}

\newcommand{\IH}{\mathcal{I}_H}
\newcommand{\Nb}{\mathtt{N}}

\newcommand{\Cinf}{\mathfrak{V}_{\mathrm{inf}}}
\newcommand{\Csup}{\mathfrak{V}_{\mathrm{sup}}}
\newcommand{\cresc}{\mathfrak{v}_{\mathrm{resc}}}
\newcommand{\hCinf}{\widehat{\mathfrak{V}}_{\mathrm{inf}}}
\newcommand{\hCsup}{\widehat{\mathfrak{V}}_{\mathrm{sup}}}
\newcommand{\hcresc}{\widehat{\mathfrak{v}}_{\mathrm{resc}}}

\newcommand{\uPG}{u_H^{\mathrm{pg}}}
\newcommand{\uCG}{u_H^{\mathrm{c}}}
\newcommand{\uNN}{u_H^{\mathrm{nn}}}

\newcommand{\norm}[1]{ \left\| #1 \right\| }

\begin{document}
%
%
\title[Neural network approximation in numerical homogenization]{
Neural network approximation of coarse-scale surrogates in numerical homogenization
}
\author[F.~Kr\"opfl, R.~Maier, D.~Peterseim]{Fabian Kr\"opfl$^\dagger$, Roland Maier$^\ddagger$, Daniel Peterseim$^*$}
\address{${}^{\dagger}$ Institute of Mathematics, University of Augsburg, Universit\"atsstr.~12a, 86159 Augsburg, Germany}
\email{fabian.kroepfl@uni-a.de}
\address{${}^{\ddagger}$ Institute of Mathematics, Friedrich Schiller University Jena, Ernst-Abbe-Platz 2, 07743 Jena, Germany}
\email{roland.maier@uni-jena.de}
\address{${}^{*}$ Institute of Mathematics \& Centre for Advanced Analytics and Predictive Sciences (CAAPS), University of Augsburg, Universit\"atsstr.~12a, 86159 Augsburg, Germany}
\email{daniel.peterseim@uni-a.de}
\date{\today}

\maketitle
%
%
\begin{abstract}
Coarse-scale surrogate models in the context of numerical homogenization of linear elliptic problems with arbitrary rough diffusion coefficients rely on the efficient solution of fine-scale sub-problems on local subdomains whose solutions are then employed to deduce appropriate coarse contributions to the surrogate model. However, in the absence of periodicity and scale separation, the reliability of such models requires the local subdomains to cover the whole domain which may result in high offline costs, in particular for parameter-dependent and stochastic problems. 		
This paper justifies the use of neural networks for the approximation of coarse-scale surrogate models by analyzing their approximation properties. For a prototypical and representative numerical homogenization technique, the Localized Orthogonal Decomposition method, we show that one single neural network is sufficient to approximate the coarse contributions of all occurring coefficient-dependent local sub-problems for a non-trivial class of diffusion coefficients up to arbitrary accuracy. We present rigorous upper bounds on the depth and number of non-zero parameters for such a network to achieve a given accuracy. Further, we analyze the overall error of the resulting neural network enhanced numerical homogenization surrogate model.
\end{abstract}


%
%
\section{Introduction}\label{sec: intro}

Surrogate models play an important role in the context of partial differential equations (PDEs) with multiscale features. After the offline adaptation to a particular coefficient, these models provide a reliable and online efficient approximation of the solution operator on some coarse scale of interest. 
In the context of modern numerical homogenization, the adaptation of the models is typically realized through coefficient-specific approximation spaces with provably optimal approximation on the coarse target scale. Prominent examples are methods such as the Localized Orthogonal Decomposition (LOD)~\cite{MalP14,HenP13,MP20}, gamblets~\cite{Owh17}, rough polyharmonic splines~\cite{OwhZB13} or generalized (multiscale) finite element methods~\cite{BabL11,EfeGW11}. For a more detailed overview of such methods, we refer to~\cite{AltHP21} and the references therein. Such numerical homogenization methods have been successful for a wide range of PDEs and their applications because they do not rely on structural assumptions on the coefficient such as periodicity or scale separation.
Another benefit of numerical homogenization approaches is their construction, which is similar to classical finite element methods. In particular, they are based on coarse-scale sub-matrices that are combined to a coarse global system matrix by appropriate local-to-global mappings.   
However, the local sub-matrices are based on the sufficiently resolved solution of fine-scale sub-problems on local subdomains (which are also known as \emph{cell problems} or \emph{corrector problems} in the homogenization literature) that resolve all features of the local PDE response. In the absence of periodicity and scale separation, the reliability of the resulting surrogate relies on the consideration of the coefficient in the full domain, i.e., the union of all local subdomains needs to cover the whole domain, which may result in high offline cost, in particular for parameter-dependent or stochastic coefficients. 

To reduce the complexity bottleneck in the process of building reliable surrogate models, there has recently been a growing interest in utilizing data-driven approaches such as deep learning in the homogenization community, see, e.g.,~\cite{ArbBSRK20,PadZ21,ChuLPZ20,WasHGKS21,ChaE18,SteSM22,HanL22,
LeuLZ22}. 
The previous work~\cite{KroMP21} proposes to replace the computation of all local system matrices -- that can each be seen as the output of a mapping from a local extract of the underlying PDE coefficient to the corresponding coarse local matrix -- by one single neural network in the offline phase of numerical homogenization. The approximation by a local and thus relatively small (trained) neural network allows one to reduce the high complexity of the corrector computations to a number of forward passes through the network, which is of particular value in non-stationary, parametric, or non-deterministic problems. Numerical experiments in~\cite{KroMP21} showed that a reasonably sized network leads to good approximation properties compared to the surrogate based on the classical computation of local sub-problems on a fine scale.

In this work, we aim at a mathematical foundation of this neural network enhanced numerical homogenization approach in the context of a prototypical elliptic model problem with an underlying (possibly fine-scale) coefficient. We restrict ourselves to the LOD which will be reviewed briefly in Section~\ref{sec: recap} below. Due to its representative nature outlined in~\cite{AltHP21}, we believe that the derivations can be generalized to other variants of classical and modern numerical homogenization methods.
From a practical point of view, the coarse-scale surrogate model computed by the LOD can be represented by a global coarse-scale system matrix that is assembled from local matrices that depend on local instances of the underlying coefficient. 
We investigate the approximation properties of neural networks as a replacement of the mappings from local coefficient to local matrix in Section~\ref{sec: nn}.

Our analysis builds upon the formal framework for studying the approximation properties of deep neural networks developed in~\cite{PetV18}. The approach is based on combining smaller neural networks as building blocks to approximate complicated nonlinear functions and has so far been used in a multitude of settings and applications. Examples include the approximation of higher-order finite elements~\cite{OpsJPS20}, Kolmogorov PDEs in the context of option pricing~\cite{ElbGJ+21}, and reduced basis methods in the context of parametric PDEs~\cite{KutPRS21}. For an overview on the expressivity of neural networks, we refer to the survey article~\cite{GuhRK20}. 

Based on these recent findings, the main result of this paper in Section~\ref{sec: nn} provides an upper bound on depth and number of parameters of an appropriate neural network which is able to achieve a certain tolerance error when approximating the local coefficient-to-matrix mappings that are the basis of the LOD. 
In Section~\ref{sec: diffSol}, we then investigate the error between discrete  solutions obtained with the original LOD approach and discrete solutions that rely on a surrogate model that is built from the local neural network approximations. In particular, we show that there exists a network such that the error of the two discrete solutions is at most of size $\mathcal{O}(H)$, where $H$ denotes the coarse scale of interest. Note that this bound is of the same error as the discretization error of the LOD with respect to the exact solution to the elliptic problem. 
Finally, we comment on the choice of fine-scale parameters for the computation and approximation of local contributions (Section~\ref{sec: intlod}) and draw conclusions.

\subsection*{Notation} Throughout this work, $C,c > 0 $ denote generic constants that are independent of the scales $H,\eps,h$, but might depend on the physical dimension $d\leq 3$. Further, we write $ \theta \lesssim \eta$ if $\theta \leq C \eta$ and $\theta \approx \eta$ if additionally also $\theta \geq c \eta$. Matrices and vectors are denoted in bold-face notation throughout the paper. In particular, we write ${\mathrm{\mathbf{Id}}}_{n}$ for the $n\times n$ identity matrix and $\mathbf{0}_n$ for the zero vector in $\mathbb{R}^{n}$.
%
%
\section{Operator Compression with Deep Neural Networks}\label{sec: recap}
In this section, we briefly explain the Localized Orthogonal Decomposition methodology and summarize the main ideas and results of~\cite{KroMP21}, which provides the basis for the theoretical analysis in the subsequent sections. The aim of this section is not to provide a detailed analysis of the method, but to give an overview of the underlying concepts.

\subsection{Setting}\label{subsec:setting}
Let $d\in\{1,2,3\}$ and and $D\subseteq\mathbb{R}^d$ a bounded, polyhedral, convex Lipschitz domain. We denote with $\V$ the standard Sobolev space of $L^2$-functions with weak derivative in $L^2(D)$ that vanish on $\partial D$. Given a function $f\in L^2(D)$ 
and a scalar coefficient $A\in L^\infty(D),$ consider the prototypical  variational problem of finding $u\in \V$ such that
\begin{equation}\label{eq:varproblem}
a(u,v) := \int_D A \nabla u \cdot\nabla v\dx = \int_D fv \dx \quad\text{for all } v \in \V.
\end{equation}
Note that we do not pose any requirements on the coefficient $A$ apart from the existence of uniform bounds $\alpha,\beta\in\mathbb{R}$ such that
\begin{equation}\label{eq:coefbounds}
0<\alpha\leq A(x) \leq \beta < \infty\quad \mathrm{for\ a.e.\ } x\in D.
\end{equation}
In particular, we do not assume periodicity or scale separation and allow for arbitrarily rough coefficients that may vary on a continuum of scales up to some fine microscale $\varepsilon$. In this setting, the Lax-Milgram theorem guarantees the existence of a unique solution $u\in \V$ to problem \eqref{eq:varproblem}. Note that $a$ and $u$ implicitly depend on $A$.

\subsection{Coarse-scale discretization by Localized Orthogonal Decomposition}\label{subsec:discretization}
Although the coefficient $A$ encodes microscopic oscillations on a scale $\varepsilon,$ in practical simulations one is often interested in the effective behavior of the solution $u$ on some macroscopic target scale of interest $H\geq \varepsilon$. It is well known, however, that the discretization of \eqref{eq:varproblem} with standard approaches such as conforming finite elements leads to acceptable results only if the problem is discretized on a scale smaller than $\varepsilon$ that fully resolves all fine-scale oscillations of the coefficient, see, e.g., the illustrations in~\cite[Ch.~2]{MP20}.  

In order to overcome this problem, numerous methods have been proposed that compress the fine-scale information contained in the coefficient $A$ to a surrogate model $\mathfrak{S}_A$ (represented by a system matrix $\mathbf{S}_A$) capable of reproducing the effective behavior of $u$ on the target scale $H$. One such method known as Localized Orthogonal Decomposition (LOD) achieves this by explicitly constructing a coarse approximation space spanned by coefficient-adapted basis functions, which can then be used to approximate \eqref{eq:varproblem} in a Galerkin fashion, resulting in a sparse system matrix $\mathbf{S}_A$ that is composed of multiple local contributions.

The LOD has originally been introduced in an elliptic setting~\cite{MalP14,HenP13}, but has been successfully applied in connection with, e.g., wave propagation problems~\cite{GalP15,Pet17,AbdH17,PetS17,GalHV18,MaiP19,RenHB19,GeeM21,LjuMP21,MaiV22,HenW20} and parabolic PDEs~\cite{MalP18,LjuMM22}. The idea has also been extended to higher-order~\cite{Mai21} and multi-resolution~\cite{HauP21} variants based on the hierarchical approach known as gamblets~\cite{Owh17}. Further, the possibility to super-localize the local sub-problems has been investigated~\cite{HauP21b}. 

In the following, we briefly describe the method and an adapted version based on local neural network approximations as presented in~\cite{KroMP21}. Note that the method in~\cite{KroMP21} may be applied to general linear divergence form partial differential equations in a straight-forward way and works beyond the LOD framework, but we restrict ourselves to its application in connection with a prototypical elliptic setting and the LOD method within this work.

Let $\calT_H$ be a uniform Cartesian mesh of $D$ with characteristic mesh size $H$, let $Q^1(\calT_H)$ be the standard first-order finite element space of piecewise polynomials with coordinate degree at most 1 and consider the conforming finite-dimensional space $V_H:= Q^1(\calT_H) \cap \V$. In our setting, $H$ denotes the length of an edge of an element of~$\calT_H$. Besides the coarse mesh $\calT_H$, we will also need a fine mesh $\calT_h$ with $h < \eps$ and the corresponding discrete space $V_h \supset V_H$, as well as an intermediate mesh $\calT_\varepsilon$ with $h<\varepsilon<H.$ Moreover, we assume these meshes to be nested and uniform refinements of each other, i.e., that $\calT_\eps$ is  a uniform refinement of $\calT_H$ and $\calT_h$ a uniform refinement of $\calT_\eps.$
An important ingredient of the method is a projective quasi-interpolation operator $\calI_H\colon \V\to \VH,$ which fulfills
\begin{equation}\label{eq:IHprop}
\|H^{-1}(v-\calI_H v)\|_{L^2(T)} + \|\nabla \calI_H v\|_{L^2(T)} \leq C \|\nabla v\|_{L^2(\Nb(T))} 
\end{equation}
for all $v\in \V$ and any element $T \in \calT_H$, where the constant $C$ is independent of $H$, and $\Nb(S) := \Nb^1(S)$ is an element patch around $S \subseteq D$ defined by
\begin{equation*}
\Nb^1(S) := \bigcup \bigl\{\overline{K} \in \calT_H\,\vert\, \overline{S} \,\cap\, \overline{K}\neq \emptyset\bigl\}.
\end{equation*} 
A prominent example of an operator $\calI_H$ that fulfills~\eqref{eq:IHprop} is the one considered in~\cite{ErnG15}, which is also used for the numerical experiments in~\cite{KroMP21}. For that reason we restrict our analysis in Section \ref{sec: nn} to this particular choice. It is defined by $\IH := E_H \circ \Pi_H$, where $\Pi_H$ denotes the piecewise $L^2$-projection onto $Q_1(\calT_H)$. For any $v_H\in Q_1(\calT_H)$, the operator $E_H$ averages in any inner node $z$ of $\calT_H$ the values of the neighboring elements, i.e.
\begin{equation}\label{eq:IH}
\big(E_H(v_H)\big)(z) := \sum_{K\in\calT_H:\atop z \in \overline{K}} \big({v_H\vert}_{\overline{K}}\big)(z)\cdot\frac{1}{\mathrm{card}\{K' \in \calT_H\,\colon\, z \in \overline{K'}\}}.
\end{equation}
Further, $\big(E_H(v_H)\big)(z) = 0$ if $z \in \partial D$. 
Given $\calI_H$, we define the so-called \emph{fine-scale space} as
\begin{equation*}
\W := \ker {\calI_H}\vert_{\Vh},
\end{equation*}
which contains functions that cannot be captured by the space $\VH$. For any $S \subseteq D$, we also define a local version of $\W$,
\begin{equation*}
\W(S) := \{w \in \W\;\vert\;\mathrm{supp}(w) \subseteq S\}.
\end{equation*}
Appropriate local versions of the fine-scale space are now used to \emph{correct} coarse-scale functions in a coefficient-dependent manner. Therefore, we iteratively define an enlarged element patch $\Nb^\ell(S) := \Nb(\Nb^{\ell-1}(S)),\,\ell \geq 2$. Let a fixed $A \in L^\infty(D)$ which fulfills \eqref{eq:coefbounds} be given. For an $\ell \in \N$ (that will later be quantified) as well as a function $v_H \in \VH$ and $T \in \calT_H$, we now introduce the \emph{element corrector} $\calQ^\ell_{T} v_H{\in \W(\Nb^\ell(T))}$ as the solution to
\begin{equation}\label{eq:corT}
a(\calQ^\ell_{T} v_H, w) = \int_T A \nabla v_H \cdot \nabla w \dx\quad  \text{for all } w \in \W(\Nb^\ell(T)).
\end{equation}
Note that this is a discrete sub-problem, which is locally computed on the fine scale $h$. The corresponding algebraic realization of this sub-problem will be discussed in Section~\ref{subsec:algebraic} below and assumptions on the fine mesh are discussed in Section~\ref{sec: intlod}. Note that the correctors $\calQ^\ell_{T} v_H$ for different $T$ are independent of each other and their respective supports are limited to $\Nb^\ell(T)$. That is, global computations on a fine scale can be avoided. Further, the correctors are only computed for a basis of $\VH$ due to linearity.

The last step towards the final multiscale method consists in defining an appropriate space based on the element correctors. Therefore, we first define the \emph{global correction operator} $\calQ^\ell\colon \VH \to \W$ by
\begin{equation}\label{eq:sumCor}
\calQ^\ell := \sum_{T \in \calT_H}\calQ^\ell_{T},
\end{equation}
which will be used to correct classical finite element functions. 
In particular, the \emph{classical LOD approximation (C-LOD)} now seeks $\uCG \in \VH$ such that 
\begin{equation}\label{eq:discvarproblemclassic}
a((\mathsf{id}-\calQ^\ell)\uCG,(\mathsf{id}-\calQ^\ell)v_H) = \int_D f\, 
v_H \dx \quad\text{for all } v_H \in \VH.
\end{equation}
Alternatively, $u$ can be approximated with a Petrov--Galerkin variant of the classical LOD. That is, the correction is only used for the trial functions but not for the test functions. 
The \emph{Petrov-Galerkin LOD approximation (PG-LOD)} $\uPG\in \VH$ solves
\begin{equation}\label{eq:discvarproblempg}
a((\mathsf{id}-\calQ^\ell)\uPG, v_H) = \int_D f\, v_H \dx \quad\text{for all } v_H \in \VH.
\end{equation}
The Petrov--Galerkin variant has several computational advantages compared to the classical method and was therefore used for the experiments in \cite{KroMP21}.

From a theoretical point of view, the approximations $\uCG$ and $\uPG$ are both first-order accurate in $L^2(D)$ if the oversampling parameter $\ell$ is chosen logarithmically in the target mesh size $H$, i.e., $\ell \approx \vert\log( H )\vert$, and the fine mesh parameter $h$ is chosen small enough (cf.~Section~\ref{sec: intlod} below for details). More precisely, it holds
\begin{equation}\label{eq:errorCLOD}
\| u - \uCG\|_{L^2(D)} \lesssim \| u - u_h\|_{L^2(D)} + (H+ e^{-c_{\mathrm{dec}}\ell}) \, \|f\|_{L^2(D)},
\end{equation}
as well as
\begin{equation}\label{eq:errorPGLOD}
\| u - \uPG\|_{L^2(D)} \lesssim \| u - u_h\|_{L^2(D)} + (H+ e^{-c_{\mathrm{dec}}\ell}) \, \|f\|_{L^2(D)},
\end{equation}
where $u_h \in \Vh$ is the classical Galerkin finite element approximation of $u$ in the space $\Vh$. 
These results follow from the triangle inequality and the results presented, e.g., in~\cite{GalP17,CaiMP20,MP20}.  
The last two estimates arise from a localization estimate of the form 
\begin{equation}\label{eq:loc}
\|\nabla(\calQ-\calQ^\ell)v_H\|_{\LL} \leq Ce^{-c_\mathrm{dec}\ell}\, \|\nabla v_H\|_{\LL}, \quad v_H \in \VH
\end{equation}
with a constant $c_\mathrm{dec}$ which is independent of $h$, $H$, and $\ell$. Note that $\calQ := \calQ^\infty$ is computed based on non-localized sub-problems~\eqref{eq:corT}. This result is, for instance, shown in~\cite{HenP13,Pet16} (based on \cite{MalP14}) and will later on be required in Section~\ref{sec: diffSol}.

\subsection{Algebraic realization of the LOD}\label{subsec:algebraic}

As explained above, the correction operator is constructed based on a number of local element correctors on given patches. In the following, we elaborate on the algebraic realization of these sub-problems. Since the corrector computations all have a similar structure, we restrict ourselves to one particular problem. Let $K \in \calT_H$ be fixed and let $\omega$ be the patch with $\ell$ additional element layers around $K$, i.e., $\omega = \Nb^\ell(K)$. The construction of the local matrix requires the solution of a corrector problem on some fine scale $h$ as introduced above.
First, we define $V_h(\omega) := \{v_h \in V_h\;\vert\;\mathrm{supp}(v_h)\subseteq \omega\}$ and $V_H\vert_\omega := \{v_H\vert_\omega \;\vert\; v_H \in V_H\}$ with nodal basis functions $\{\lambda_j\}_{j=1}^{n_\omega}$ and $\{\Lambda_j\}_{j=1}^{N_\omega}$, respectively. Here, $n_\omega = \dim \Vh(\omega)$ and $N_\omega = \dim \VH\vert_\omega$.
For a given basis function $\Lambda_i$, the corrector problem seeks the solution $\calQ_{K}^\ell (\Lambda_{i}\vert_K) \in \Vh(\omega)$ to 
\begin{equation}\label{eq:corProbT}
a(\calQ_{K}^\ell (\Lambda_{i}\vert_K),w_h) = a(\Lambda_i\vert_K,w_h)
\end{equation}
for all $w_h \in \Vh(\omega) \cap \ker \IH$. Note that we only need to compute the corrections for the $2^d$ basis functions for which $\mathrm{supp}\,\Lambda_i \cap K \neq \emptyset$. For simplicity and without loss of generality, we assume in the following that this is the case for the first indices $1,\ldots,2^d$.

To avoid defining a basis of $\Vh(\omega) \cap \ker \IH$, it is favorable to write~\eqref{eq:corProbT} as an equivalent saddle point problem, which reads 
\begin{equation}\label{eq:saddlepoint}
\begin{aligned}
a(\calQ_{K}^\ell (\Lambda_{i}\vert_K),  v_h)\qquad\, &+& (\varphi_i, \IH v_h)_{L^2(\omega)} \quad&=\quad a(\Lambda_{i}\vert_K, v_h),\\
(\IH\calQ_{K}^\ell (\Lambda_{i}\vert_K),\mu_H)_{L^2(\omega)} && &=\quad 0
\end{aligned}
\end{equation}  
for all $v_h \in \Vh(\omega)$ and $\mu_H \in \VH\vert_\omega$, where $\varphi_i \in \VH\vert_\omega$ is the associated Lagrange multiplier. Note that~\eqref{eq:saddlepoint} has a unique solution pair $(\calQ_{T}^\ell\Lambda_{i},\varphi_i)$ and the problem is equivalent to solving~\eqref{eq:corProbT}; see, for instance, \cite[Ch.~2]{Mai20}.

To solve problem~\eqref{eq:saddlepoint} numerically, it is reformulated as a system of linear equations. We use the linear combinations 
\begin{align*}
\calQ_{K}^\ell (\Lambda_{i}\vert_K) = \sum_{j = 1}^{n_\omega} \xi_j^i \lambda_j,\qquad\qquad
\varphi_i = \sum_{j = 1}^{N_\omega} \phi_j^i \Lambda_j,
\end{align*}
and write $\mathbf{\Xi}^i = [\xi_1^i,\ldots,\xi_{n_\omega}^i]^T$ and $\mathbf{\Phi}^i = [\phi_1^i,\ldots,\phi_{N_\omega}^i]^T$ for the corresponding vectors.

Let $\mathbf{S} := \mathbf{S}_{A,\omega}$ be the finite element stiffness matrix with respect to $\Vh(\omega)$ (weighted by the coefficient $A$ and with built-in boundary conditions) and $\mathbf{P}_\omega \in \R^{n_\omega \times N_\omega}$ (respectively $\mathbf{P}_{\omega,K} \in \R^{n_\omega \times {2^d}}$) the prolongation matrix between functions in $\VH\vert_\omega$ and $\Vh(\omega)$ (respectively $\VH\vert_K$ and $\Vh(\omega)$). Further, $\mathbf{I}_\omega \in \R^{N_\omega \times n_\omega}$ denotes the realization of the operator $\IH$ on $\omega$ (without explicit boundary conditions). 
With these matrices, \eqref{eq:saddlepoint} can be expressed as 
\begin{equation}\label{eq:saddlepointMatrix}
\begin{aligned}
\mathbf{S}\mathbf{\Xi}\,\, +\,\, \mathbf{I}_\omega^T\mathbf{\Phi} \quad&=\quad \mathbf{S} \mathbf{P}_{\omega,K},\\
\mathbf{I}_\omega \mathbf{\Xi}\quad &=\quad \mathbf{0}_{N_\omega \times 2^d},
\end{aligned}
\end{equation}
where $\mathbf{\Xi} = [\mathbf{\Xi}^1|\ldots|\mathbf{\Xi}^{2^d}]$ and $\mathbf{\Phi} = [\mathbf{\Phi}^1|\ldots|\mathbf{\Phi}^{2^d}]$ are matrices that comprise the vectors $\{\mathbf{\Xi}^i\}$ and $\{\mathbf{\Phi}^i\}$, respectively.
Some calculations show that $\mathbf{\Phi} = (\mathbf{I}_\omega \mathbf{S}^{-1} \mathbf{I}_\omega^T)^{-1} \mathbf{I}_\omega \mathbf{P}_{\omega,K}$.
Inserting this in the first equation of~\eqref{eq:saddlepointMatrix} yields
\begin{equation*}
\mathbf{\Xi} = \mathbf{P}_{\omega,K} - \mathbf{S}^{-1} \mathbf{I}_\omega^T(\mathbf{I}_\omega \mathbf{S}^{-1} \mathbf{I}_\omega^T)^{-1} \mathbf{I}_\omega \mathbf{P}_{\omega,K}.
\end{equation*}
The local Petrov--Galerkin LOD matrix with respect to the element $K$ and the patch $\omega$ around $K$ is defined by
\begin{equation}\label{eq:locLOD}
\mathbf{S}_{\omega}^\mathrm{pg}[i,j] = a(\Lambda_{i},(\mathsf{id}-\calQ_{K}^\ell)(\Lambda_{j}\vert_K)), \quad (i,j) \in \{1,\ldots,N_\omega\} \times \{ 1,\ldots,2^d\}.
\end{equation}
In terms of the matrix system~\eqref{eq:saddlepointMatrix}, $\mathbf{S}_{\omega}^\mathrm{pg}$ can be equivalently defined by
\begin{equation}\label{eq:SLOD}
\begin{aligned}
\mathbf{S}_{\omega}^\mathrm{pg} &= \mathbf{P}_\omega^T \mathbf{S} (\mathbf{P}_{\omega,K}-\mathbf{\Xi}) = 
\mathbf{P}_\omega^T \mathbf{S} \mathbf{S}^{-1} \mathbf{I}_\omega^T (\mathbf{I}_{\omega} \mathbf{S}^{-1} \mathbf{I}_{\omega}^T)^{-1}\mathbf{I}_\omega \mathbf{P}_{\omega,K} \\&= 
\mathbf{P}_\omega^T \mathbf{I}_{\omega}^T (\mathbf{I}_{\omega} \mathbf{S}^{-1} \mathbf{I}_{\omega}^T)^{-1}\mathbf{I}_\omega \mathbf{P}_{\omega,K},
\end{aligned}
\end{equation}
which will be an important representation in order to analyze the approximability by a suitable neural network in the following. 
Note that $\mathbf{S}_{\omega}^\mathrm{pg}$ implicitly depends on the coefficient~$A$ through the bilinear form $a$ and the correction operator $\calQ_{K}^\ell$. 

\subsection{Neural network approximation of the local sub-problems}
\label{sec:nnsub}
The method presented in~\cite{KroMP21} can be used to approximate the local contributions to the LOD stiffness matrix, i.e., the matrices $\mathbf{S}_{\omega}^\mathrm{pg}$ of the previous subsections, by one single neural network that maps from a local instance of the coefficient on $\omega := \Nb^\ell(K)$ to an approximation of the matrix $\mathbf{S}_{\omega}^\mathrm{pg}$. For completeness, we briefly discuss the main ideas and refer to~\cite[Sec.~4]{KroMP21} for numerical experiments that show the feasibility of the approach. In the subsequent section, we also rigorously investigate the strategy from the viewpoint of approximation theory.

The matrices $\mathbf{S}_{\omega}^\mathrm{pg}$, as defined in~\eqref{eq:locLOD}, are rectangular matrices that are characterized by the local degrees of freedom $i = 1,\ldots,N_\omega$ and $j = 1,\ldots,2^d$. For the final Petrov--Galerkin stiffness matrix corresponding to the discretized problem~\eqref{eq:discvarproblempg}, we require the sum of the local corrections $\calQ_{K}^\ell$; cf.~\eqref{eq:sumCor}. That is, to build the stiffness matrix, we need the sum of the local matrices $\mathbf{S}_{\omega}^\mathrm{pg}$ as well. This is done with appropriate local-to-global mappings $\Phi_K$ that transform the local matrix for an element $K \in \calT_H$ to an inflated matrix with respect to the global degrees of freedom. The global stiffness matrix $\mathbf{S}^\mathrm{pg}$ thus has the form
\begin{equation}\label{eq:decSLOD}
\mathbf{S}^\mathrm{pg} = \sum_{K \in \calT_H} \Phi_K\big(\mathbf{S}_{\Nb^\ell(K)}^\mathrm{pg}\big),
\end{equation}
see~\cite[Sec.~3.4]{KroMP21} for further details. As mentioned above, the local matrices $\mathbf{S}_{\Nb^\ell(K)}^\mathrm{pg}$ implicitly depend on $A$, but the mappings $\Phi_K$ are independent of the coefficient. The method for approximating the matrix $\mathbf{S}^\mathrm{pg}$ now keeps this coefficient-independent decomposition and only replaces the local matrices $\mathbf{S}_{\Nb^\ell(K)}^\mathrm{pg}$ by suitable approximations based on a single neural network. This is a reasonable approach due to the fact that the local corrector problems~\eqref{eq:corProbT} all have a similar structure and are (up to the actual values of the coefficient) independent of their position within the domain. 

In the following, we theoretically justify this strategy in the sense that we show that there exists a suitable network that provides a (local) surrogate such that the solution computed with the corresponding global stiffness matrix is reasonably close to the actual PG-LOD solution. In particular, we are interested in the dimensions (i.e., depth and number of non-zero parameters) of such a neural network. In Section~\ref{sec: nn}, we first focus on the approximability of the local LOD sub-problems by a neural network, before we investigate the difference between the corresponding discrete global coarse-scale solutions in Section~\ref{sec: diffSol}.
%
%
\section{Neural Network Approximation of Local LOD Matrices}\label{sec: nn}
In this section, we investigate the complexity of a neural network to approximate the mapping from a coefficient to the local PG-LOD matrix as given in~\eqref{eq:SLOD} and give rigorous upper bounds for depth and number of non-zero parameters of such a network. 
To start with, we have to establish a framework for neural network approximation theory that we will then apply to our specific setting. 

\subsection{Neural network calculus} \label{sec: nnc}
We adopt the main ideas and definitions of the formal framework developed in \cite{PetV18} and quickly recap the results from \cite{KutPRS21} regarding the approximation of matrix inversion by deep ReLU-neural networks. 
\begin{definition}[Neural network]
A \emph{neural network} $\Psi$ of \emph{depth} $L\in \mathbb{N}$ is a sequence of matrix-vector tuples 
\begin{equation*}
\Psi = ((\mathbf{W}_l, \mathbf{b}_l))_{l=1}^L,
\end{equation*}
where each \emph{layer} $(\mathbf{W}_l, \mathbf{b}_l)$ consists of a \emph{weight matrix} $\mathbf{W}_l \in \mathbb{R}^{N_l \times N_{l-1}}$ and a \emph{bias vector} $\mathbf{b}_l\in \mathbb{R}^{N_l}$ for $N_0, \dots, N_L \in \mathbb{N}$. 

We will refer to $L(\Psi) : = L$ as the \emph{number of layers}, to $\mathrm{dim}_{\mathrm{in}}(\Psi) := N_0$ as the \emph{input dimension} and to $\mathrm{dim}_{\mathrm{out}}(\Psi) := N_L$ as the \emph{output dimension} of $\Psi$. We call $M_l(\Psi) := \|\mathbf{W}_l \|_0 + \|\mathbf{b}_l \|_0$ the \emph{number of parameters in the $l$-th layer}, where $\| \cdot \|_0$ counts the number of non-zero entries in a given matrix or vector. The total \emph{number of parameters} is then given by $M(\Psi) : = \sum_{l = 1}^L M_l$. 
\end{definition}

\begin{definition}[Realization of a neural network]
Let $\Psi$ be a neural network of depth~$L$. For a set $S\subset \mathbb{R}^{N_0}$ and an \emph{activation function} $\rho\colon \mathbb{R} \rightarrow \mathbb{R}$ that is assumed to act component-wise on vectors by convention, the \emph{realization of $\Psi$ with activation function $\rho$ over $S$} is the mapping $\calR_\rho^S(\Psi)\colon S \rightarrow \mathbb{R}^{N_L}$ implemented by the neural network $\Psi$. It is given by
\begin{equation*}
\calR_\rho^S(\Psi)(\mathbf{x}) := \mathbf{x}_L,
\end{equation*}
where $\mathbf{x}_L$ results from
\begin{equation*}
\begin{aligned}
\mathbf{x}_0 &:= \mathbf{x}, \\
\mathbf{x}_{l} &:= \rho(\mathbf{W}_{l}\, \mathbf{x}_{l-1} + \mathbf{b}_l), \quad l = 1,\dots,L-1, \\
\mathbf{x}_L &:= \mathbf{W}_{L}\, \mathbf{x}_{L} + \mathbf{b}_L,
\end{aligned}
\end{equation*}
i.e.,
\begin{equation*}
\calR_\rho^S(\Psi)(x) = \mathbf{W}_L\, \rho(\mathbf{W}_{L-1}(\dots \rho(\mathbf{W}_{2}\, \rho(\mathbf{W}_1 \mathbf{x} + \mathbf{b}_1)+\mathbf{b}_2)\dots) + \mathbf{b}_{L-1}) + \mathbf{b}_L.
\end{equation*}
\end{definition}
Although many different activation functions $\rho$ are conceivable and used in practice, we restrict ourselves to one of the most popular choices, the so-called \emph{Rectified Linear Unit (ReLU)} activation function given by $\rho(x) := \max \{0,x\}$ in this work.
\begin{definition}[Concatenation of neural networks] Let $\Psi^1 : = ((\mathbf{W}_l^1, \mathbf{b}^1_l))_{l = 1}^{L_1}$ and $\Psi^2 : = ((\mathbf{W}_l^2, \mathbf{b}^2_l))_{l = 1}^{L_2}$ be two neural networks of depth $L_1, L_2$, respectively, such that $\mathrm{dim}_{\mathrm{in}}(\Psi^1) = \mathrm{dim}_{\mathrm{out}}(\Psi^2).$ Then the \emph{concatenation of $\Psi^1$ and $\Psi^2$} is denoted by $\Psi^1 \bullet \Psi^2$ and reads \\
\small
\begin{equation*}
\Psi^1 \bullet \Psi^2 := ((\mathbf{W}_1^2,\mathbf{b}_1^2), \dots, (\mathbf{W}_{L_2-1}^2,\mathbf{b}_{L_2-1}^2), (\mathbf{W}_1^1 \mathbf{W}_{L_2}^2, \mathbf{W}_1^1 \mathbf{b}_{L_2}^2 + \mathbf{b}_1^1), (\mathbf{W}_2^1,\mathbf{b}_2^1), \dots, (\mathbf{W}_{L_1}^1,\mathbf{b}_{L_1}^1)). \\
\end{equation*}
\normalsize
\end{definition}
It is easy to check that the realization of this concatenation implements the composition of the realizations of $\Psi^1$ and $\Psi^2,$ i.e.,
\begin{equation*}
\calR_\rho^{\mathbb{R}^{N_0^2}}(\Psi^1 \bullet \Psi^2) = \calR_\rho^{\mathbb{R}^{N_0^1}}(\Psi^1) \circ \calR_\rho^{\mathbb{R}^{N_0^2}}(\Psi^2).
\end{equation*}
Thus, the concatenation of two neural networks can be interpreted as connecting them in series. The problem with this approach is that, in general, $M(\Psi^1 \bullet \Psi^2)$ cannot be bounded linearly with respect to $M(\Psi^1)$ and $M(\Psi^2)$. The next lemma shows one possible exception to this.
\begin{lemma}\label{lem:permnn}
Let $\Psi^1$ be a neural network of the form
\begin{equation*}
\Psi^1 = \big((\mu \mathbf{Q},\mathbf{0}_{n})\big), \quad \mu \in\mathbb{R},
\end{equation*}
where $\mathbf{Q}\in\mathbb{R}^{n\times n}$ is a permutation matrix and $\Psi^2$ a neural network with output dimension $n$. Then it holds that 
\begin{itemize}
\item[(i)] $L(\Psi^1 \bullet \Psi^2)=L(\Psi^2),$
\item[(i)] $M(\Psi^1 \bullet \Psi^2)=M(\Psi^2).$
\end{itemize}
\end{lemma}
\begin{proof}
This is a special case of \cite[Lemma A.1.]{KutPRS21}.
\end{proof}

In the general case of more complicated networks, however, we need to introduce another type of concatenation, which is built on the construction of a two-layer neural network that implements the identity function.
\begin{lemma}[{\cite[Lem.~2.3]{PetV18}}]\label{lem:idnn}
Let $n\in\mathbb{N}$ and define the two-layer neural network
\begin{equation*}
\Psi^{\mathrm{Id}}_n := \Bigg( \Big( \begin{bmatrix}\mathrm{\bf Id}_n \\ -\mathrm{\bf Id}_n \end{bmatrix}, \mathbf{0}_{2n}\Big) , \Big( \begin{bmatrix}\mathrm{\bf Id}_n,\, -\mathrm{\bf Id}_n \end{bmatrix} , \mathbf{0}_{n} \Big) \Bigg)
\end{equation*}
with input and output dimension $n$. Then it holds that
\begin{equation*}
R^{\mathbb{R}^n}_\rho(\Psi^{\mathrm{Id}}_n) = \mathrm{Id}_{\mathbb{R}^n},
\end{equation*}
i.e., $\Psi^{\mathrm{Id}}_n$ implements the identity function on $\mathbb{R}^n$.
\end{lemma}

With this definition, we can introduce the so-called sparse concatenation that allows us to precisely control the number of non-zero parameters when two or more networks are connected in series.
\begin{definition}[Sparse concatenation of neural networks] 
Given two networks $\Psi^1$ and $\Psi^2$ of depth $L_1, L_2 \in \mathbb{N}$, respectively, with $\mathrm{dim}_{\mathrm{in}}(\Psi^1) = \mathrm{dim}_{\mathrm{out}}(\Psi^2) =: n$ as above, we define the \emph{sparse concatenation of $\Psi^1$ and $\Psi^2$} as
\begin{equation*}
\Psi^1 \odot \Psi^2 := \Psi^1 \bullet \Psi_{n}^{\mathrm{Id}}  \bullet \Psi^2,
\end{equation*}
where $\Psi_{n}^{\mathrm{Id}}$ is the identity network from Lemma \ref{lem:idnn}.
\end{definition}
Obviously, the sparse concatenation does not change the function realized by the network compared to the regular concatenation and therefore also implements the composition of the individual realizations. Moreover, the following lemma shows the desired linear bound on $M(\Psi^1 \odot \Psi^2)$. 
\begin{lemma}[{cf.~\cite[Lem.~5.3]{ElbGJ+21}}]\label{lem:sparseccnn}
Given two neural networks $\Psi^1$ and $\Psi^2,$ their sparse concatenation fulfills
\begin{itemize}
\item[(i)] $L(\Psi^1 \odot \Psi^2) \leq L(\Psi^1) + L(\Psi^2),$
\item[(ii)] $M(\Psi^1 \odot \Psi^2) \leq M(\Psi^1) + M(\Psi^2) + M_1(\Psi^1) + M_{L(\Psi^2)}(\Psi^2) \leq 2M(\Psi^1) + 2M(\Psi^2)\\ \hspace*{2.1cm} \lesssim M(\Psi^1) + M(\Psi^2).$ 
\end{itemize}
Moreover, given $k>2$ neural networks $\Psi^1,\dots,\Psi^k,$ their sparse concatenation satisfies
\begin{itemize}
\item[(i)] $L(\Psi^1 \odot \dots \odot \Psi^k) \leq \sum_{i=1}^k L(\Psi^i),$
\item[(ii)] $M(\Psi^1 \odot \dots \odot \Psi^k) \lesssim \sum_{i=1}^k M(\Psi^i).$ 
\end{itemize}
\end{lemma}

In order to construct more complex networks out of smaller building blocks, we also need to be able to connect them in parallel. The definition below describes how this can be done for networks with identical depths. In principle, network parallelization can also be defined for networks of different depths. However, since we do not require such a construction for our proofs in the next subsection, we restrict ourselves to the simpler case. 

\begin{definition}[Parallelization of neural networks]
Let $\Psi^i = ((\mathbf{W}_l^i, \mathbf{b}_l^i))_{l = 1}^L,\ i = 1,\dots,k$ be a family of $k$ neural networks with identical input dimension $n$ and identical depth $L$. Then the \emph{parallelization} of the $\Psi^i$ is given by
\begin{equation*}
P(\Psi^1,\dots,\Psi^k) = \left( 
\left(
\begin{bmatrix}
\mathbf{W}_1^1 & & \\ & \ddots & \\ & & \mathbf{W}_1^k
\end{bmatrix},
\begin{bmatrix}
\mathbf{b}_1^1 \\ \vdots \\ \mathbf{b}_1^k
\end{bmatrix}
\right),
\dots,
\left(
\begin{bmatrix}
\mathbf{W}_L^1 & & \\ & \ddots & \\ & & \mathbf{W}_L^k
\end{bmatrix},
\begin{bmatrix}
\mathbf{b}_L^1 \\ \vdots \\ \mathbf{b}_L^k
\end{bmatrix}
\right)
\right).
\end{equation*}
\end{definition}
For the realizations of parallelized networks of identical input dimension $n$, it holds that 
\begin{equation*}
R_\rho^{\mathbb{R}^n}(\Psi^1,\dots,\Psi^k)(\mathbf{x_1},\dots,\mathbf{x_k}) = [R_\rho^{\mathbb{R}^n}(\Psi^1)(\mathbf{x_1}),\dots,R_\rho^{\mathbb{R}^n}(\Psi^k)(\mathbf{x_k})]
\end{equation*}
for $\mathbf{x_1},\dots,\mathbf{x_k}\in\mathbb{R}^n$. 
\begin{lemma}\label{lem:parann} Let $\Psi^1,\dots,\Psi^k$ be a family of $k$ neural networks with identical input dimension and identical depth $L$. Then the parallelization $P(\Psi^1,\dots,\Psi^k)$ fulfills 
\begin{itemize}
\item[(i)] $L(P(\Psi^1,\dots,\Psi^k)) = L,$
\item[(ii)] $M(P(\Psi^1,\dots,\Psi^k)) = \sum_{i = 1}^k M(\Psi^i).$
\end{itemize}
\end{lemma}
\begin{proof}
The result follows by construction.
\end{proof}

Now that we have laid down the foundations of neural network calculus, we turn to the question of how to approximate matrix inversion with a neural network. Since in the standard formulation of feedforward networks the input and output are one-dimensional arrays, we choose to consider columnwise ``flattened'' (or vectorized) matrices. For $\mathbf{A}\in\mathbb{R}^{k\times l}$ we thus write
\begin{equation*}
\text{vec}(\mathbf{A}) := [A_{1,1},\dots,A_{k,1},\dots,A_{1,l},\dots,A_{k,l}]^T \in \mathbb{R}^{k\, \cdot\, l}
\end{equation*} 
and, conversely, for $\mathbf{v} = [v_{1,1},\dots,v_{k,1},\dots,v_{1,l},\dots,v_{k,l}]^T\in\mathbb{R}^{k\, \cdot\, l},$
\begin{equation*}
\text{mat}(\mathbf{v}) := (v_{i,j})_{i,j = 1}^{k,l} \in \mathbb{R}^{k\times l}.
\end{equation*}
Furthermore, we define for $n\in\mathbb{N}$ and $Z>0$ the set
\begin{equation*}
K_{n}^Z := \left\{\mathrm{vec}(\mathbf{A})\; \vert\; \mathbf{A}\in\mathbb{R}^{n\times n}, \|\mathbf{A} \|_2 \leq Z \right\},
\end{equation*}
where $\|\cdot\|_2$ denotes the spectral norm induced by the Euclidean vector norm. 
The idea behind the approximation of matrix inversion with a deep neural network is based on approximating the Neumann series of matrices which can be suitably bounded in terms of their spectral norm. That is, we approximate the map 
\begin{equation*}
\begin{aligned}
\mathrm{inv}: \{\mathbf{A}\in\mathbb{R}^{n\times n} \;\vert\; \| \mathbf{A}\|_2\leq 1-\delta \} & \rightarrow \mathbb{R}^{n\times n},\qquad 
\mathbf{A} &\mapsto (\mathrm{\mathbf{Id}}_{n} -\mathbf{A})^{-1} = \sum_{k=0}^\infty \mathbf{A}^k.
\end{aligned}
\end{equation*}
The following theorem makes this precise and shows that matrix inversion can be approximated up to arbitrary precision.
\begin{theorem}[{\cite[Thm.~3.8]{KutPRS21}}]\label{theo:invnn}
For $n\in \mathbb{N}, \vartheta\in (0,1/4)$ and $\delta\in(0,1),$ let
\begin{equation*}
m(\vartheta,\delta) := \left\lceil \frac{\log(0.5 \vartheta\delta)}{\log(1-\delta)} \right\rceil.
\end{equation*}
Then there exists a neural network $\Psi_{\mathrm{inv},\vartheta}^{1-\delta, n}$ with input dimension $n^2$ and output dimension $n^2$ with
\begin{itemize} 
\item[(i)] $L(\Psi_{\mathrm{inv}, \vartheta}^{1-\delta,n}) \lesssim \log(m(\vartheta,\delta)) \, \big(\log(1/\vartheta) +  \log(m(\vartheta,\delta)) + \log(n)\big)$,  
\item[(ii)] $M(\Psi_{\mathrm{inv}, \vartheta}^{1-\delta,n}) \lesssim m(\vartheta,\delta)\, \log^2(m(\vartheta,\delta))\,n^3 \, \big(\log(1/\vartheta) +  \log(m(\vartheta,\delta)) + \log(n)\big)$, 
\end{itemize} 
such that for all $\mathrm{vec}(\mathbf{A})\in K_n^{1-\delta}$ it holds that
\begin{itemize}
\item[(iii)] $\sup_{\mathrm{vec}(\mathbf{A})\in K_{n}^{1-\delta}} \Big\|\, (\mathrm{\mathbf{Id}}_{n} -\mathbf{A})^{-1} - \mathrm{mat}\Big(R_\rho^{K_{n}^{1-\delta} } (\Psi_{\mathrm{inv},\vartheta}^{1-\delta, n})(\mathrm{vec}(\mathbf{A})) \Big)\, \Big\|_2 \leq \vartheta,$
\item[(iv)] $\Big\|\mathrm{mat}\Big(R_\rho^{K_{n}^{1-\delta} } (\Psi_{\mathrm{inv},\vartheta}^{1-\delta, n})(\mathrm{vec}(\mathbf{A})) \Big)\, \Big\|_2 \leq \vartheta + \left\|(\mathrm{\mathbf{Id}}_{n} -\mathbf{A})^{-1} \right\|_2 \leq \vartheta + \frac1\delta.$ 
\end{itemize}
\end{theorem}

\begin{remark}[Conservation of symmetry in neural network matrix inversion]\label{rem:symnn}
In the original construction of the network $\Psi_{\mathrm{inv},\vartheta}^{1-\delta, n}$ given in \cite{KutPRS21}, the network is not symmetry-preserving in the sense that if $\mathbf{A}$ is a symmetric matrix, the output matrix 
\begin{equation*}
\mathrm{mat}\Big(R_\rho^{K_{n}^{1-\delta} } \big(\Psi_{\mathrm{inv},\vartheta}^{1-\delta, n}\big)(\mathrm{vec}(\mathbf{A}))\Big)
\end{equation*}
might not be symmetric although $(\mathrm{\textbf{Id}}_n - \mathbf{A})^{-1}$ is. For our intended application below, this is not very convenient and makes the error analysis much more difficult. However, the construction can be slightly adapted to preserve the symmetry. Using our notation, define in the induction step of the proof of \cite[Proposition A.4.]{KutPRS21} the modified network 
\begin{equation*}
\widetilde{\Psi}^{1/2,n}_{k,\vartheta} := \widetilde{\Psi}^{1,n,n,n}_{\mathrm{mult},\vartheta/4}\; \odot P\left(\Psi^{1/2,n}_{2^j,\vartheta}, \Psi^{1/2,n}_{t,\vartheta} \right) \bullet \left(\left(\begin{bmatrix}\mathrm{\bf Id}_{n^2} \\ \mathrm{\bf Id}_{{n^2}} \end{bmatrix}, \mathbf{0}_{2n^2}\right)\right),
\end{equation*}
where $\widetilde{\Psi}^{1,n,n,n}_{\mathrm{mult},\vartheta/4}$ is a network that is constructed analogously to the network implementing general matrix-matrix multiplication in \cite[Proposition 3.7]{KutPRS21}. However, instead of the matrix $\mathbf{D}$ in the proof therein, consider a modified matrix $\widetilde{\mathbf{D}}$, which for $i,j,k\in \{1,\dots, n\}$ and $\mathbf{A}, \mathbf{B}\in \mathbb{R}^{n\times n}$ is defined by
\begin{equation*}
\widetilde{\mathbf{D}}_{i,k,j}(\mathrm{vec}(\mathbf{A}), \mathrm{vec}(\mathbf{B})) := \begin{cases}
(\mathbf{A}_{i,k}, \mathbf{B}_{k,j}), \quad \mathrm{if}\ i\leq j, \\
(\mathbf{B}_{i,k}, \mathbf{A}_{k,j}), \quad \mathrm{if}\ i > j.
\end{cases}
\end{equation*}
Then the desired result follows analogously to~\cite{KutPRS21}.
\end{remark}
Note that the result in Theorem~\ref{theo:invnn} above does not utilize any structural properties of the matrix to be inverted such as sparsity, positive definiteness or symmetry. If one has additional structural knowledge about the matrix, as in our intended application in the sections below, we believe that by approximating more sophisticated solvers, for example multigrid approaches, with a neural network, one could improve upon the cubic complexity in the number of non-zero parameters.

The previous theorem is the key ingredient to proving that the approximation of the local LOD matrices can be realized by deep neural networks, to which we turn now.

\subsection{Application to local PG-LOD matrices}\label{sec:approxlocal}
Before we formulate the central result of this section, we have to introduce some more notation. We define 
\begin{equation*}
\mathfrak{A}:= \{A\in L^\infty(D)\, \vert\, 0<\alpha\leq A(x) \leq \beta <\infty\ \mathrm{for\; a.e.\; } x\in D\},
\end{equation*}
which is the set of admissible coefficients for problem~\eqref{eq:varproblem} according to~\eqref{eq:coefbounds}, where now $\alpha, \beta$ are fixed constants. Since neural networks take only discrete information as input, we restrict ourselves to the finite-dimensional subset 
\begin{equation*}
\mathfrak{A}_\varepsilon := \{A\in \mathfrak{A}\, \vert\, A\vert_T\ \mathrm{is\ constant\ for\ every\ } T\in\calT_\varepsilon\} \subseteq \mathfrak{A}
\end{equation*}
of element-wise constant coefficients on $\calT_\varepsilon$. Recall that the meshes $\calT_\varepsilon$ and $\calT_h$ are assumed to be uniform refinements of $\calT_H$ and that for a given patch $\omega = \Nb^\ell(K),\ K\in\calT_H$, the local matrices $\mathbf{S}_{\omega}^\mathrm{pg}$ from~\eqref{eq:locLOD} that we want to approximate solely depend on the values of $A$ on $\omega$.

Note that we implicitly assume that the patch $\omega$ has ``full'' size, i.e., that it corresponds to an element $K$ that is at least $\ell+1$ layers of coarse elements away from the boundary of~$D$. 
We emphasize, however, that the treatment of the boundary cases can as well be incorporated into the theory below by changing the overall complexity of the network only by a moderate factor. The first ingredient is that patches closer to the boundary can be artificially extended by elements where the coefficient has the value zero to keep a uniform size of all patches; we refer to~\cite[Sec.~3.4]{KroMP21} for further details. Further, a slight adaptation of the matrices used in the algebraic realization of the method presented in Section~\ref{subsec:algebraic} is required for the boundary cases. This adaptation is conceptually straightforward but rather technical, since one has to account for artificial ``degrees of freedom'' outside the physical domain $D$. For better readability, we thus only focus on the inner patches in the proof of Theorem \ref{theo:lodnn} below.

We enumerate the elements in the restricted mesh $\calT_\varepsilon(\omega)$ by  $1,\dots,m_\ell:=((2\ell+1)H/\varepsilon)^d$, which allows us to store the respective values of the coefficient in a vector $\mathbf{A}_{\varepsilon,\omega} \in \R^{m_\ell}$. Based on this, we consider the set
\begin{equation*}
\mathfrak{A}_{\varepsilon,\omega}:= \{\mathbf{A}_{\varepsilon,\omega}\, \vert\, A\in\mathfrak{A}_\varepsilon\} \subseteq \mathbb{R}^{m_\ell}
\end{equation*}
of all possible input vectors corresponding to the patch $\omega$.
Moreover, we enumerate the inner nodes of $\calT_h(\omega)$ by $1,\dots,n_\ell:=((2\ell+1)H/h-1)^d$ and inner and boundary nodes of $\calT_H(\omega)$ by $1,\dots,N_\ell:=(2\ell+2)^d$.

\begin{theorem}[Neural network approximation of local PG-LOD matrices]\label{theo:lodnn}
Let $\eta\in (0,1/4)$ be a given error tolerance and let $\ell \in \N$.
Then there exists a neural network $\Psi^{\mathrm{pg}}_\eta$ with input dimension $m_\ell$ and output dimension $N_\ell2^d$ such that for any $\omega = \Nb^\ell(K),\ K\in\calT_H$, and any $\mathbf{A}_{\varepsilon,\omega}\in\mathfrak{A}_{\varepsilon,\omega}$ it holds that 
\begin{equation*}
\Big\|\mathbf{S}_{\omega}^\mathrm{pg} - \mathrm{mat}\Big(\calR^{\mathfrak{A}_{\varepsilon,\omega}}_\rho(\Psi^{\mathrm{pg}}_\eta)(\mathbf{A}_{\varepsilon,\omega})\Big)\Big\|_2 \leq \eta.
\end{equation*}
The network $\Psi^{\mathrm{pg}}_\eta$ has the following properties,
\begin{itemize}
\item[(i)] $L(\Psi^{\mathrm{pg}}_\eta) \lesssim \log(m(\theta,\delta)) \, \big(\log(1/\theta) +  \log(m(\theta,\delta)) + \log(n_{\ell})\big) \\
\hspace*{1.2cm}+ \log(m(\gamma,\widehat{\delta})) \, \big(\log(1/\gamma) +  \log(m(\gamma,\widehat{\delta})) + \log(N_{\ell})\big) + 1,$
\item[(ii)] $M(\Psi^{\mathrm{pg}}_\eta) \lesssim m(\theta,\delta) \log^2(m(\theta,\delta))n_{\ell}^3 \, \big(\log(1/\theta) +  \log(m(\theta,\delta)) + \log(n_{\ell})\big)
\\ \hspace*{1.4cm} + m(\gamma,\widehat{\delta}) \log^2(m(\gamma,\widehat{\delta}))N_{\ell}^3 \, \big(\log(1/\gamma) +  \log(m(\gamma,\widehat{\delta})) + \log(N_{\ell})\big) + (\ell H/h)^{2d},$
\end{itemize}
where $\delta \approx \lambda_{\mathrm{min}}(\mathbf{S})\,\lambda_{\mathrm{max}}(\mathbf{S})^{-1}$ and $\widehat{\delta} \approx \lambda_{\mathrm{min}}(\mathbf{I}_\omega \mathbf{S}^{-1}\mathbf{I}_\omega^T)\,\lambda_{\mathrm{max}}(\mathbf{I}_\omega \mathbf{S}^{-1}\mathbf{I}_\omega^T)^{-1}$. Further, the values $\theta,\gamma\in (0,\eta)$ are chosen such that
\begin{equation*}
\theta<\min\Bigg\{\lambda_{\mathrm{min}}(\mathbf{\mathbf{I}_\omega \mathbf{S}^{-1}\mathbf{I}_\omega^T)},\frac{\lambda_{\mathrm{min}}(\mathbf{I}_\omega \mathbf{S}^{-1}\mathbf{I}_\omega^T)}{2\,\cresc\|\mathbf{I}_\omega^T \|_2^2}\Bigg\},
\end{equation*}
and
\begin{equation*}
\frac{\cresc\norm{\mathbf{P}_{\omega,K}}_2\norm{\mathbf{P}_{\omega}}_2\norm{\mathbf{I}_{\omega}}_2^4}{\hCinf^2}\, \theta + \norm{\mathbf{P}_{\omega,K}}_2\norm{\mathbf{P}_{\omega}}_2\norm{\mathbf{I}_{\omega}}_2^2 \gamma \leq \eta
\end{equation*}
with $\cresc \approx  \lambda_{\mathrm{max}}(\mathbf{S})^{-1}$ and $\hCinf \approx \lambda_{\mathrm{min}}(\mathbf{I}_\omega \mathbf{S}^{-1}\mathbf{I}_\omega^T)$; cf.~also the precise definitions in~\eqref{eq:cresc} and \eqref{eq:specY}, \eqref{eq:spechatY} in the proof.
\end{theorem}
\begin{proof}
The main idea of the proof is to start from the representation \eqref{eq:SLOD} of the local PG-LOD stiffness matrix and subsequently implement or approximate all matrix multiplications or inversions of the mapping $\mathbf{A}_{\varepsilon,\omega} \mapsto \mathbf{S}_{\omega}^\mathrm{pg}$ by suitably constructed neural networks. The matrix multiplications can be implemented exactly by single-layer networks, the inversions, however, have to be approximated using Theorem \ref{theo:invnn}. In the end, these individual building blocks can be connected in series to obtain a final network with the desired properties.
A detailed proof can be found in Section \ref{app:lodnn} of the Appendix.
\end{proof}

We close the section by bounding the size of a network in order for it to well-approximate any local PG-LOD stiffness matrix with an error of size $\mathcal{O}(H^k)$, $k \in \N,$ which will be essential for the error estimation of the global errors between the two surrogate models that are based on deterministic local PG-LOD matrices and their network approximations, respectively.

\begin{corollary}\label{cor:network}
Let $H<1/4$ and $\ell \approx |\log(H)|$. For any $k\in\mathbb{N},$ there exists a neural network $\Psi^\mathrm{pg}_{H^k}$ with input dimension $m_\ell$, and output dimension $N_\ell2^d$ such that for any $\omega = \Nb^\ell(K),\ K\in\calT_H$, and any $\mathbf{A}_{\varepsilon,\omega}\in\mathfrak{A}_{\varepsilon,\omega}$
\begin{equation*}
\Big\|\mathbf{S}_{\omega}^\mathrm{pg} - \mathrm{mat}\Big(\calR^{\mathfrak{A}_{\varepsilon,\omega}}_\rho(\Psi^{\mathrm{pg}}_{H^k})(\mathbf{A}_{\varepsilon,\omega})\Big)\Big\|_2 \lesssim H^k.
\end{equation*}
Further, we have
\begin{itemize}
\item[(i)] $L(\Psi^{\mathrm{pg}}_{H^k}) \lesssim |\log(h)|^2 + |\log(k)|^2$,
\item[(ii)] $M(\Psi^{\mathrm{pg}}_{H^k}) \lesssim h^{-2}\Big(k |\log(H)| + |\log(h)|\Big)\Big(|\log(h)|^3 + |\log(k)|^3\Big) (\vert \log(H)\vert H/h)^{3d}$.
\end{itemize}
\end{corollary}

\begin{proof}
The result is obtained by applying Theorem \ref{theo:lodnn} above with $\eta = H^k$ and estimating all quantities in the resulting upper bounds on depth and number of non-zero parameters in terms of the mesh sizes $H,h$.
For more details, see the full proof in Section \ref{app:network} of the Appendix.
\end{proof}

\begin{remark}
Note that the goal of the above results is to show that a suitable neural network exists that allows one to approximate the local PG-LOD matrices up to arbitrary accuracy. We emphasize that this does not guarantee that such a network is actually learned during a training phase. Further, the bounds on the number of layers and non-zero parameters of the network have to be understood as worst-case estimates. 
\end{remark}
%
%
\section{Error Analysis of the Neural Network Enhanced Surrogate Model}\label{sec: diffSol}
In the previous section, we have seen that the local PG-LOD matrices individually can be well-approximated by a neural network. However, in applications we are usually more concerned about how this translates to the resulting difference of the respective global solutions when applied to our test problem. This involves multiple error sources that will be discussed in this section. 

In the following, we investigate the error in solutions between the solution $\uPG \in \VH$ of~\eqref{eq:discvarproblempg}, i.e., the Petrov--Galerkin LOD solution, and the solution $\uNN\in\VH$ obtained from the stiffness matrix which is assembled based on the (optimal) local neural network according to Theorem~\ref{theo:lodnn} for a suitable tolerance $\eta$. Let $\mathbf{u}^\mathrm{pg}$ be the vector corresponding to $\uPG$, i.e.,
\begin{equation}\label{eq:vec}
\uPG = \sum_{j=1}^{N} \mathbf{u}^\mathrm{pg}_j\Lambda_j.
\end{equation}
Here, $\{\Lambda_j\}_{j=1}^{N}$ denote the nodal basis functions of $\VH$ with $N = \dim \VH$. As above, we write 
\begin{equation}\label{eq:decS}
\mathbf{S}^\mathrm{pg} = \sum_{K \in \calT_H} \Phi_K\big(\mathbf{S}_{\Nb^\ell(K)}^\mathrm{pg}\big),
\end{equation}
cf.~\eqref{eq:decSLOD}. In the same manner, let $\mathbf{u}^\mathrm{nn}$ be the vector corresponding to $\uNN$ and
\begin{equation}\label{eq:decT}
\mathbf{S}^\mathrm{nn} := \sum_{K \in \calT_H} \Phi_K\big(\mathbf{\Theta}_K \big)
\end{equation} 
the respective stiffness matrix. Note that the matrices 
\begin{equation}\label{eq:Theta}
\mathbf{\Theta}_K := \mathrm{mat}\big(\calR^{\mathfrak{A}_{\varepsilon,\Nb^\ell(K)}}_\rho(\Psi^{\mathrm{pg}}_\eta)(\mathbf{A}_{\varepsilon,\Nb^\ell(K)})\big)
\end{equation}
are the local matrices obtained by a forward pass of the local instance $\mathbf{A}_{\varepsilon,\Nb^\ell(K)}$ of a fixed coefficient $A$ on $\Nb^\ell(K)$ through the neural network, cf.~Section~\ref{sec: nnc}, particularly Theorem~\ref{theo:lodnn} and Corollary~\ref{cor:network}. Our goal is to investigate under which conditions the error $\|\uPG - \uNN\|_{L^2(D)}$ is bounded by some given tolerance, ideally of order $\mathcal{O}(H)$. This would be optimal since $\uPG$ already includes an error which scales at least like $\mathcal{O}(H)$ compared to the ideal solution to~\eqref{eq:varproblem}. Note that we have the following norm equivalence (see, e.g.,~\cite{Fri73} and observe that the eigenvalues of the local mass matrix are bounded by $H^d\,6^{-d}$ and $H^d\,2^{-d}$ from below and above, respectively), 
\begin{equation}\label{eq:normEquiv}
\big(H/6\big)^{d} \|\mathbf{v} \|_{2}^2 \leq \mathbf{v}^T \mathbf{M} \mathbf{v} = \| v \|_{L^2(D)}^2 \leq H^d \|\mathbf{v} \|_{2}^2,
\end{equation}
where $v \in \VH$ and $\mathbf{v}$ is the corresponding vector.
Here, $\mathbf{M}$ denotes the classical finite element mass matrix. 
Therefore, it holds that
\[
\|\uPG - \uNN\|_{L^2(D)}^2 = (\mathbf{u}^\mathrm{pg}-\mathbf{u}^\mathrm{nn})^T \mathbf{M} (\mathbf{u}^\mathrm{pg}-\mathbf{u}^\mathrm{nn}) \approx H^d \|\mathbf{u}^\mathrm{pg}-\mathbf{u}^\mathrm{nn}\|_2^2.
\]
To investigate the error between these two solutions, it is favorable to consider the symmetric version of the LOD for an intermediate step. This will be treated in the following subsection.

\subsection{Difference between C-LOD and PG-LOD approximation}\label{ss:diffCGPG}

Let $\uCG \in \VH$ be the solution to~\eqref{eq:discvarproblemclassic} with corresponding vector $\mathbf{u}^\mathrm{c}$ as above. The stiffness matrix to~\eqref{eq:discvarproblemclassic} is denoted $\mathbf{S}^\mathrm{c}$. The next lemma shows that the difference $\uCG - \uPG$, respectively $\mathbf{u}^\mathrm{c}-\mathbf{u}^\mathrm{pg}$, exponentially decays with increasing localization parameter $\ell$ introduced in Section~\ref{subsec:discretization}. 

\begin{lemma}\label{lem:CGvsPG}
Let $\uCG \in \VH$ and $\uPG \in \VH$ be the solutions of~\eqref{eq:discvarproblemclassic} and~\eqref{eq:discvarproblempg}, respectively, and let $\mathbf{u}^\mathrm{c}$ and $\mathbf{u}^\mathrm{pg}$ be the corresponding vectors, cf.~\eqref{eq:vec}. Then
\begin{equation*}
\|\uCG- \uPG\|_{L^2(D)} \lesssim \exp(-c\ell)
\end{equation*}
and
\begin{equation*}
\|\mathbf{u}^\mathrm{c} - \mathbf{u}^\mathrm{pg}\|_{2} \lesssim H^{-d/2}\exp(-c\ell),
\end{equation*}
where the constant hidden in $\lesssim$ depends on the right-hand side $f$.
\end{lemma}

\begin{proof} 
The proof is based on the observation that the error between the solutions to ~\eqref{eq:discvarproblemclassic} and~\eqref{eq:discvarproblempg} can be reduced to a decay estimate of the form~\eqref{eq:loc}. 
The full proof is stated in Section \ref{app:CGvsPG} of the Appendix.
\end{proof}

This result will allow us to switch from the PG-LOD approximation to the C-LOD approximation in the next subsection. Note that the C-LOD approximation is much more convenient to use in connection with spectral estimates. In particular, we have the following lemma. 
\begin{lemma}\label{lem:boundR}
Let $\mathbf{S}^\mathrm{c}$ be the stiffness matrix corresponding to~\eqref{eq:discvarproblemclassic}. Its minimal eigenvalue fulfills 
\begin{equation*}
\lambda_\mathrm{min}(\mathbf{S}^\mathrm{c}) \geq c H^{d}
\end{equation*}
with a constant $c$ that does not depend on the mesh size $H$.
\end{lemma}
\begin{proof}
The proof is given in Section~\ref{app:boundR} of the Appendix.
\end{proof}

\subsection{Error in Solution}\label{sec:errorsolution}

With the preliminary considerations of the previous subsection, we can now investigate the error between the solutions $\uPG$ and $\uNN$. 

\begin{theorem}[Coarse-scale error]\label{th:error} Let $H <1/4$, $\ell \approx |\log(H)|$. There exists a neural network $\Psi^\mathrm{pg}$ with
\begin{equation}\label{eq:boundsNetwork} 
L(\Psi^\mathrm{pg}) \lesssim |\log(h)|^2\quad \text{ and }\quad M(\Psi^\mathrm{pg}) \lesssim h^{-2}\,|\log(h)|^4 \,(\vert \log(H)\vert H/h)^{3d}
\end{equation}
such that for any $A \in \mathfrak{A}_\varepsilon$ the solutions $\uPG \in \VH$ of~\eqref{eq:discvarproblempg} (and its vector representation~$\mathbf{u}^\mathrm{pg}$) and the network solution $\uNN$ (respectively the vector $\mathbf{u}^\mathrm{nn}$) fulfill
\begin{equation*}
\| \mathbf{u}^\mathrm{pg} - \mathbf{u}^\mathrm{nn}\|_2 \lesssim H^{1-d/2} \quad\text{ and }\quad \|\uPG - \uNN\|_{L^2(D)} \lesssim H.
\end{equation*}
\end{theorem}

\begin{proof}
The proof reduces the stated error to local contributions that can be estimated using Corollary~\ref{cor:network}. To employ useful properties of the symmetric and positive definite matrix~$\mathbf{S}^\mathrm{c}$, we take an intermediate step to first estimate the error compared to the symmetric solution $\uCG \in \VH$ of~\eqref{eq:discvarproblemclassic} (respectively the vector $\mathbf{u}^c$) and make use of Lemma~\ref{lem:CGvsPG} and the eigenvalue bound in Lemma~\ref{lem:boundR}.
The detailed proof is presented in Section~\ref{app:error}.
\end{proof}

Theorem~\ref{th:error} shows that there exists a neural network such that the approach of Section~\ref{sec: recap} leads to a global coarse-scale error between the discrete PG-LOD solution and its neural network-based variant of order $\mathcal{O}(H)$. The size of such a network can be bounded dependent on the scales $H$ and $h$. This leads to an overall error compared to the exact solution to the elliptic problem of the order $\mathcal{O}(H)$ as well, provided that $h$ is reasonably small, see also~\eqref{eq:errorPGLOD}. In certain cases, the necessary choice of the scale $h$ (and thus the dependence of the size of the network) can be stated in terms of $H$ and $\eps$ only as investigated in the following section. 
%
%
\section{How Fine Is Enough? - LOD on an Appropriate Fine Scale}\label{sec: intlod}

In the previous sections, we have investigated how PG-LOD stiffness matrices can be well-approximated by a local neural network in the sense that the error of the (global) coarse discretizations are reasonably close with respect to the mesh size $H$. Note that the dimensions of the network depend on the scale $h$ on which the corrector problems~\eqref{eq:corProbT} are computed. 
In this section, we want to investigate how fine this fine scale actually needs to be. 

\subsection{Finding an optimal fine computational scale}

In general, the corrections for the LOD may be computed on an arbitrarily fine scale $h$. Choosing a finer $h$ also means smaller discretization errors and we generally have 
\begin{equation*}
\|u - u_h\|_{L^2(D)} \to 0 \quad \text{ as } h \to 0
\end{equation*}
for the first term on the right-hand side of~\eqref{eq:errorPGLOD}.
For the corrector problems, however, a finer $h$ also increases the size of the linear systems that need to be solved. As investigated in Section~\ref{sec: nnc}, this leads to more parameters in the neural network which approximates the local contributions. 
Since the network itself already introduces an error and the LOD also comes with an error of order $\mathcal{O}(H)$, we now seek the largest possible scale $h$ with $0 < h < H$ such that the scale $h$ is sufficient to still obtain an overall error of order $H$ while minimizing the necessary parameters of the network. This particularly means that the term $\|u - u_h\|_{L^2(D)}$ should be of order $\mathcal{O}(H)$. We have the following result.
\begin{lemma}[Fine-scale error]\label{lem:fineError}
There exists an $s >0$ (depending on $\alpha,\beta$) such that for any $A \in \mathfrak{A}_\eps$, the solution $u$ to~\eqref{eq:varproblem} and its Galerkin approximation $u_h$ in~$\Vh$ satisfy $u \in H^{1+s}(D)$ and
\begin{equation}\label{eq:fineErr}
\| \nabla (u - u_h) \|_{L^2(D)} \lesssim h^s \|u \|_{H^{1+s}(D)}.
\end{equation}
\end{lemma}

\begin{proof}
The regularity results presented in \cite[Ch.~2]{Pet01} state that there exists some $s > 0$ such that $u \in H^{1+s}(D)$ for our class of coefficients $\mathfrak{A}_\varepsilon$. 
With an appropriate interpolation operator $\calI_h$, we may thus derive (see, e.g.,~\cite[Thm.~6.4]{ErnG15})
\begin{equation*}
\| \nabla (u - u_h) \|_{L^2(D)} \lesssim \| \nabla (1 - \calI_h) u \|_{L^2(D)} \leq Ch^s \|u\|_{H^{1+s}(D)}. \qedhere
\end{equation*} 
\end{proof}

\begin{remark}[Dependence on $\eps$]
If $A \in W^{1,\infty}(D)$ with oscillations on the scale $\eps$, i.e., $\|A\|_{W^{1,\infty}(D)} \leq C\eps^{-1}$, we have $s=1$ in Lemma~\ref{lem:fineError} with $\|u\|_{H^2(D)} \lesssim h/\eps \|f\|_{L^2(D)}$ (see, e.g., the bounds in the proofs of~\cite[Lem.~4.3]{PetS12} or~\cite[Lem.~3.3]{MaiP19}). That is, the choice $h \approx H \eps$ leads to an error of size $\mathcal{O}(H)$.

If instead $A \in \mathfrak{A}_\eps$, the condition on $h$ reads  
$h \approx H^{1/s}\,\|u\|_{H^{1+s}(D)}^{-1/s}$. We emphasize that the norm of $u$ does not depend on $h$ and 
one can expect that $\|u\|_{H^{1+s}(D)} = \mathcal{O}(\eps^{-s})$, leading to the optimal choice $h \approx H^{1/s}\eps$.
\end{remark}

With the above considerations and Theorem~\ref{th:error}, we can finally state a result that quantifies the worst-case size and depth of a (local) neural network to achieve an overall error of the corresponding global surrogate of order $\mathcal{O}(H)$ independently of the scale $h$ (which, in practice, could be chosen arbitrarily small).

\begin{corollary}
Let the solution $u$ of~\eqref{eq:varproblem} fulfill $u \in H^{1+s}(D)$ and $\|u\|_{H^{1+s}(D)} = \mathcal{O}(\eps^{-s})$ for some $s>0$. Then, there exists a local neural network $\Psi^\mathrm{pg}$ with depth
\begin{equation*}
L(\Psi^\mathrm{pg}) \lesssim s^{-2}|\log(H)|^2 + |\log(\varepsilon)|^2
\end{equation*} 
and size
\begin{equation*}
M(\Psi^\mathrm{pg}) \lesssim H^{-2/s}\varepsilon^{-2}\,\Big(s^{-4}|\log(H)|^4 + |\log(\varepsilon)|^4\Big) \,\big(\vert \log(H)\vert H^{1-1/s}\varepsilon^{-1}\big)^{3d}
\end{equation*}
such that the error of $u$ and the solution $\uNN\in\VH$ obtained from the stiffness matrix which is assembled based on $\Psi^\mathrm{pg}$ fulfills
\[
\| u - \uNN \|_{L^2(D)} \lesssim H.
\]
\end{corollary}
%
%
\section{Conclusion and Outlook}\label{sec: conclusion}
In this paper, we have theoretically investigated the approximation properties of neural networks to approximate the coarse-scale contributions of local sub-problems in numerical homogenization. The computation of these local problems that are solved on a fine scale is the bottleneck when computing reliable coarse-scale surrogates in numerical homogenization. Therefore, replacing this process by a single trained neural network allows for a significant speedup, because the computation of an appropriate surrogate is reduced to simple forward passes through the network. We have focused on the Localized Orthogonal Decomposition method as a representative numerical homogenization method and presented upper bounds on the size of a network that leads to a total error between the deterministic discrete approximation and its variant based on the output surrogate of a trained neural network that scales linearly with respect to the coarse scale of interest. 

We emphasize that the presented results focus on general approximation properties and do not provide an answer to the important question regarding whether and how optimal approximating networks can be found through training from data, which might not be possible at all for certain tasks as recently shown in~\cite{ColAH21}. Nevertheless, our findings provide insight into how the dimensions of a suitable neural network for the approximation of the local sub-problems in numerical homogenization should be chosen if the target scale of interest and the oscillation scale of the coefficient are given. This serves as a first step in developing mathematically rigorous guidelines on how to design suitable neural network architectures for numerical homogenization tasks. Further investigations into this direction are subject to future research.

\section{Acknowledgments}
F.~Kr\"opfl and D.~Peterseim acknowledge funding by the Deutsche Forschungsgemeinschaft within the research project \emph{Autonomous research for exploring structure-property linkages and optimizing microstructures} (PE 2143/7-1).

\newcommand{\etalchar}[1]{$^{#1}$}

\appendix

\section{Proofs in Section \ref{sec: nn}} \label{app:nn}

\subsection{Proof of Theorem \ref{theo:lodnn}} \label{app:lodnn}
As mentioned at the beginning of Section \ref{sec:approxlocal}, we assume that $K\in\calT_H$ is an element that is at least $\ell+1$ layers of elements away from the boundary of $D$ and set $\omega=\Nb^\ell(K)$.
The idea of the proof is to start from the representation of the local PG-LOD matrix given in~\eqref{eq:SLOD}, namely 
\begin{equation}
\mathbf{S}_{\omega}^\mathrm{pg} = 
\mathbf{P}_\omega^T \mathbf{I}_{\omega}^T (\mathbf{I}_{\omega} \mathbf{S}^{-1} \mathbf{I}_{\omega}^T)^{-1}\mathbf{I}_\omega \mathbf{P}_{\omega,K}, 
\end{equation}
and decomposing the approximation of the mapping $\mathbf{A}_{\varepsilon,\omega} \mapsto \mathrm{vec}(\mathbf{S}_{\omega}^\mathrm{pg})$ into the following seven consecutive steps:
\begin{itemize}
\item[1)] Linear transformation $\mathbf{A}_{\varepsilon,\omega} \mapsto \mathrm{vec}(\mathbf{S})$, 
\item[2)] Inversion $\mathrm{vec}(\mathbf{S})\mapsto \mathrm{vec}(\mathbf{S}^{-1})$, 
\item[3)] Linear transformation $\mathrm{vec}(\mathbf{S}^{-1})\mapsto \mathrm{vec}(\mathbf{S}^{-1} \mathbf{I}_\omega^T) =:\mathrm{vec}(\mathbf{X})$,
\item[4)] Linear transformation $\mathrm{vec}(\mathbf{X})\mapsto \mathrm{vec}(\mathbf{I}_\omega \mathbf{X}) =:\mathrm{vec}(\mathbf{Y})$,
\item[5)] Inversion $\mathrm{vec}(\mathbf{Y})\mapsto \mathrm{vec}(\mathbf{Y}^{-1})$,
\item[6)] Linear transformation $\mathrm{vec}(\mathbf{Y}^{-1})\mapsto \mathrm{vec}(\mathbf{Y}^{-1} \mathbf{I}_\omega \mathbf{P}_{\omega,K})=:\mathrm{vec}(\mathbf{Z})$,
\item[7)] Linear transformation $\mathrm{vec}(\mathbf{Z}) \mapsto \mathrm{vec}(\mathbf{P}_\omega^T \mathbf{I}_\omega^T \mathbf{Z}).$
\end{itemize}
Note that in fact steps 1, 3, 4, 6, and 7 can be implemented exactly by single-layer neural networks. The matrix inversions in Steps 2 and 5, however, have to be approximated up to some tolerance.
Our goal is therefore to construct neural networks $\Psi^1, \dots, \Psi^7$, which implement and approximate these steps. Connecting those building blocks in series then yields a network $\Psi^{\mathrm{pg}}_\eta$ with the desired properties.\\

\emph{Step 1:} The key insight in the first step is that the mapping $\mathbf{A}_{\varepsilon,\omega}\mapsto \mathrm{vec}(\mathbf{S})$ can be written as a linear transformation of the input vector $\mathbf{A}_{\varepsilon,\omega}\mapsto \mathbf{U} \mathbf{A}_{\varepsilon,\omega}.$ Based on this observation, the single-layer network 
\begin{equation*}
\Psi^1 := \big((\mathbf{U}, \mathbf{0}_{n_\ell^2})\big) 
\end{equation*}
with input dimension $m_\ell$ and output dimension $n_{\ell}^2$ exactly implements the desired map. 
Writing $\lambda_i,\, i=1,\dots, n_{\ell},$ for the classical nodal basis of $V_h(\omega),$ and identifying the index of entry $i\in\{1,\dots,n_\ell^2\}$ in $\mathrm{vec}(\mathbf{S})$ with the index pair $(k,l)\in \{1,\dots,n_\ell\} \times \{1,\dots,n_\ell\}$ of the corresponding entry in $\mathbf{S}$, the matrix $\mathbf{U}\in\mathbb{R}^{n_\ell^2 \times m_\ell}$ is given by
\begin{equation*}
\mathbf{U}_{i,j} = 
\int_{T_j} \nabla\lambda_{k}\cdot \nabla\lambda_{l} \, \dx,
\end{equation*}
where $T_j$ is the $j$-th element in $\calT_\varepsilon(\omega).$
The fact that $\norm{\mathbf{U}}_0 \leq 2^{3d-1}m_{\ell}(\varepsilon/h)^d$ by a rough estimation then yields 
\begin{itemize}
\item[(i)] $L(\Psi^1) = 1,$
\item[(ii)] $M(\Psi^1) \leq 2^{3d-1}m_{\ell}(\varepsilon/h)^d \lesssim ((\ell+1)H/h)^d \lesssim (\ell H/h)^d$. \\
\end{itemize}

\emph{Step 2:} To approximate the inversion $\mathrm{vec}(\mathbf{S})\mapsto \mathrm{vec}(\mathbf{S}^{-1}),$ we utilize Theorem \ref{theo:invnn} to construct a suitable network $\Psi^2$. In order to do so, however, we first have to rescale the input $\mathrm{vec}(\mathbf{S})$ with a value $\cresc$ in such a way that $\| \mathrm{\mathbf{Id}}_{n_{{\scalebox{.45}{$\ell$}}}^2} - \cresc\mathbf{S}\|_2 \leq 1-\delta$ for some $\delta\in(0,1)$. Since all coefficients under consideration are bounded from below and above by $\alpha$ and $\beta$, respectively, there exist optimal values $\Cinf, \Csup$ (dependent on $h$, $\alpha$, and $\beta$) such that
\begin{equation*}
0 < \Cinf \leq \inf_{\mathbf{v}\in\mathbb{R}^{n_{\scalebox{.45}{$\ell$}}}\setminus\{0\}} \frac{\mathbf{v}^T \mathbf{S} \mathbf{v}}{\mathbf{v}^T\mathbf{v}} \leq \sup_{\mathbf{v}\in\mathbb{R}^{n_{\scalebox{.45}{$\ell$}}}\setminus\{0\}} \frac{\mathbf{v}^T \mathbf{S} \mathbf{v}}{\mathbf{v}^T\mathbf{v}} \leq \Csup < \infty,
\end{equation*}
for any $A\in\mathfrak{A}$. That is, the spectrum of $\mathbf{S}$ is always contained in $[\Cinf, \Csup]$. Choosing 
\begin{equation}\label{eq:cresc}
\cresc \in (0, \Csup^{-1})
\end{equation} 
and setting $\delta:= \cresc \Cinf,$ the symmetry of $\mathbf{S}$ then implies
\begin{equation*}
\big\| \mathrm{\mathbf{Id}}_{n_{{\scalebox{.45}{$\ell$}}}^2} - \cresc\mathbf{S}\big\|_2 \leq \vert 1-\cresc \Cinf\vert = 1-\delta.
\end{equation*}
The step of rescaling the input corresponds to feeding it through the one-layer network $\big((\cresc\, \mathrm{\mathbf{Id}}_{n_{{\scalebox{.45}{$\ell$}}}^2}, \mathbf{0}_{n_{\scalebox{.45}{$\ell$}}^2})\big).$ Moreover, it can be easily seen that if $((\mathbf{W},\mathbf{b}))$ is the neural network that implements $\cresc\mathrm{vec}(\mathbf{S}),$ then $((-\mathbf{W},-\mathbf{b} + \mathrm{vec}(\mathrm{\mathbf{Id}}_{n_{\scalebox{.45}{$\ell$}}^2}))$ implements $\mathrm{vec}(\mathrm{\mathbf{Id}}_{n_{\scalebox{.45}{$\ell$}}^2} - \cresc\mathbf{S})$. For any $\theta\in (0,\eta)$, Theorem \ref{theo:invnn} then guarantees the existence of a neural network $\Psi_{\mathrm{inv}, \theta}^{1-\delta, n_{{\scalebox{.45}{$\ell$}}}^{2}}$ such that 
\begin{equation} \label{eq:step2}
\Big\|\, \left(\cresc\mathbf{S}\right)^{-1} - \mathrm{mat}\Big(R_\rho^{K_{n_{\scalebox{.45}{$\ell$}}}^{1-\delta}} \big(\Psi_{\mathrm{inv},\theta}^{1-\delta, n_{{\scalebox{.45}{$\ell$}}}^2}\big)\left(\mathrm{vec}(\mathrm{\mathbf{Id}}_{n_{\scalebox{.45}{$\ell$}}^2} - \cresc\mathbf{S})\right) \Big)\, \Big\|_2 \leq \theta.
\end{equation}
Lastly, after the approximate inversion is performed, one has to scale the output back to the original scaling to obtain an approximation to $\mathrm{vec}(\mathbf{S}^{-1})$ rather than $(1/\cresc)\mathrm{vec}(\mathbf{S}^{-1})$. This is again done with the one-layer network $\big((\cresc\, \mathrm{\mathbf{Id}}_{n_{{\scalebox{.45}{$\ell$}}}^2}, \mathbf{0}_{n_{{\scalebox{.45}{$\ell$}}}^2})\big).$ Taking those operations together, we define 
\begin{equation*}
\Psi^2 := \big((\cresc\, \mathrm{\mathbf{Id}}_{n_{{\scalebox{.45}{$\ell$}}}^2}, \mathbf{0}_{n_{\ell}^2})\big) \bullet \Psi_{\mathrm{inv}, \eta}^{1-\delta, n_{{\scalebox{.45}{$\ell$}}}^2} \odot \big((-\cresc\, \mathrm{\mathbf{Id}}_{n_{{\scalebox{.45}{$\ell$}}}^2},\mathrm{vec}(\mathrm{\mathbf{Id}}_{n_{{\scalebox{.45}{$\ell$}}}^2}))\big).
\end{equation*}
Combining Theorem \ref{theo:invnn} with Lemma \ref{lem:permnn} and Lemma \ref{lem:sparseccnn}, we obtain the following bounds for the inversion network $\Psi^2$, 
\begin{itemize}
\item[(i)] $L(\Psi^2) \lesssim \log(m(\theta,\delta)) \, \big(\log(1/\theta) +  \log(m(\theta,\delta)) + \log(n_{\ell})\big) + 1, $
\item[(ii)] $M(\Psi^2) \lesssim m(\theta,\delta) \log^2(m(\theta,\delta))n_{\ell}^3 \, \big(\log(1/\theta) +  \log(m(\theta,\delta)) + \log(n_{\ell})\big) + n_{\ell}^2.$ \\
\end{itemize}

\emph{Step 3:} Similarly to Step 1, multiplying the neural network approximation of $\mathbf{S}^{-1}$ with~$\mathbf{I}_\omega^T$ from the right can be implemented by a single-layer neural network. Since both~$\mathbf{S}^{-1}$ and its approximation are symmetric matrices (see Remark \ref{rem:symnn}), it holds for both matrices that $\mathbf{S}^{-1} \mathbf{I}_\omega^T = (\mathbf{I}_\omega \mathbf{S}^{-1})^T$ and therefore $\mathrm{vec}(\mathbf{S}^{-1} \mathbf{I}_\omega^T) = \mathbf{Q}\, \mathrm{vec}(\mathbf{I}_\omega \mathbf{S}^{-1})$ for a suitable permutation matrix $\mathbf{Q}$ that maps $\mathrm{vec}(\mathbf{M})$ to $\mathrm{vec}(\mathbf{M}^T)$ for an arbitrary matrix $\mathbf{M}\in\mathbb{R}^{N_{\ell}\times n_{\ell}}$. To implement the map $\mathrm{vec}(\mathbf{S}^{-1})\mapsto \mathrm{vec}(\mathbf{I}_\omega \mathbf{S}^{-1})$, we consider the parallelization of $n_{\ell}$ identical copies $\Psi^3_1, \dots ,\Psi^3_{n_{\ell}}$ of the single-layer network 
\begin{equation*}
\Psi_{\mathbf{I}_\omega} := \big((\mathbf{I}_\omega, \mathbf{0}_{N_\ell}) \big).
\end{equation*}
Concatenating the output of this parallelization with the one-layer network $((\mathbf{Q}, \mathbf{0}_{n_{\ell}\cdot N_{\ell}}))$, i.e., setting
\begin{equation*}
\Psi^3 := ((\mathbf{Q}, \mathbf{0}_{n_{\ell}\cdot N_{\ell}}))\bullet P(\Psi_1^3, \dots ,\Psi^3_{n_{\ell}}),
\end{equation*}
then exactly implements the desired transformation. If the interpolation operator $\calI_H$ given in \eqref{eq:IH} is used, the corresponding matrix $\mathbf{I}_\omega\in\mathbb{R}^{N_{\ell}\times n_{\ell}}$ is a sparse matrix that satisfies $\|\mathbf{I}_\omega \|_0 \leq N_{\ell}(2H/h+1)^d$. Moreover, since $\mathbf{Q}$ is a permutation matrix, it does not change the number of non-zero parameters when concatenated with the parallelization network due to Lemma \ref{lem:permnn}. With Lemma \ref{lem:parann}, we finally obtain the following complexity estimates,
\begin{itemize}
\item[(i)] $L(\Psi^3) = 1,$
\item[(ii)] $M(\Psi^3) = n_{\ell}N_{\ell}(2H/h+1)^d = (((2\ell+1)H/h - 1)(2\ell+2)(2H/h + 1))^d \\ \hspace*{1.2cm} \lesssim (\ell^2H/h (H/h + 1))^d\lesssim(\ell H/h)^{2d}$. \\
\end{itemize} 

\emph{Step 4:} Analogous to the previous step, we consider $N_{\ell}$ identical copies $\Psi^4_1, \dots ,\Psi^4_{N_{\ell}}$ of the linear transformation network
$\Psi_{\mathbf{I}_\omega} = \big((\mathbf{I}_\omega, \mathbf{0}_{N_\ell}) \big)$ 
that implements multiplication of an input vector with $\mathbf{I}_\omega$. Then define the parallelization
\begin{equation*}
\Psi^4 := P(\Psi_1^4, \dots ,\Psi^4_{N_{\ell}}),
\end{equation*}
which implements the desired transformation exactly. Note that in this step, no permutation of the output of the parallelization is necessary. We obtain the same bounds as in the previous step, i.e.,
\begin{itemize}
\item[(i)] $L(\Psi^4) = 1,$
\item[(ii)] $M(\Psi^4) \lesssim (((2\ell+1)H/h - 1)(2\ell+2)(2H/h + 1))^d \lesssim (\ell H/h)^{2d}.$ \\
\end{itemize}

\emph{Step 5:} This step is similar to Step 2 and utilizes again Theorem \ref{theo:invnn} to approximate the inversion of $\mathbf{Y} := \mathbf{I}_\omega \mathbf{S}^{-1} \mathbf{I}_\omega^T \in \mathbb{R}^{N_{\ell}\times N_{\ell}}$. However, we have to consider that at this stage the approximation 
\begin{equation*}
\widehat{\mathbf{Y}}:= \mathbf{I}_\omega \cresc\, \mathrm{mat}\Big(R_\rho^{K_{n_{{\scalebox{.45}{$\ell$}}}}^{1-\delta}} \big(\Psi_{\mathrm{inv},\theta}^{1-\delta, n_{{\scalebox{.45}{$\ell$}}}^2}\big)(\mathrm{vec}(\mathrm{\mathbf{Id}}_{n_{{\scalebox{.45}{$\ell$}}}^2} - \cresc\mathbf{S}))\Big) \mathbf{I}_{\omega}^T
\end{equation*}
instead of the true matrix $\mathbf{Y}$ will be given as the input to the network to be constructed due to the inexact inversion of Step 2 above. Note that $\mathbf{Y}$ is a symmetric and positive definite matrix (cf.~also the proof of Corollary \ref{cor:network}) and observe that under the additional condition
\begin{equation*}
\theta<\min\Bigg\{\lambda_{\mathrm{min}}(\mathbf{Y}), \frac{\lambda_{\mathrm{min}}(\mathbf{Y})}{2\,\cresc\|\mathbf{I}_\omega \|_2^2}\Bigg\},
\end{equation*}
the same statement holds true for $\widehat{\mathbf{Y}}$, since
\begin{equation*}
\vert\lambda_{\mathrm{min}}(\mathbf{Y})-\lambda_{\mathrm{min}}(\widehat{\mathbf{Y}})\vert \leq \|\mathbf{Y}-\widehat{\mathbf{Y}}\|_2 \leq \cresc \|\mathbf{I}_\omega \|^2_2\, \theta < 0.5\,\lambda_{\mathrm{min}}(\mathbf{Y}) 
\end{equation*}
and therefore 
\begin{equation}\label{eq:eigBound1}
\lambda_{\mathrm{min}}(\widehat{\mathbf{Y}}) = \lambda_{\mathrm{min}}(\mathbf{Y}) - \big(\lambda_{\mathrm{min}}(\mathbf{Y})-\lambda_{\mathrm{min}}(\widehat{\mathbf{Y}})\big) > 0.5\,\lambda_{\mathrm{min}}(\mathbf{Y}).
\end{equation}
Further, we have 
\begin{equation}\label{eq:eigBound2}
\lambda_{\mathrm{max}}(\widehat{\mathbf{Y}}) = \|\widehat{\mathbf{Y}}\|_2 \leq \|\mathbf{Y}\|_2 + \|\mathbf{Y} - \widehat{\mathbf{Y}}\|_2 \leq \lambda_{\mathrm{max}}(\mathbf{Y}) + 0.5\, \lambda_{\mathrm{min}}(\mathbf{Y})   < 2\,\lambda_{\mathrm{max}}(\mathbf{Y}).
\end{equation}
This implies the existence of optimal values $\hCinf,\,\hCsup$ (dependent on $H$, $h$, $\alpha$, and $\beta$) such that
\begin{equation}\label{eq:specY}
0 < \hCinf \leq \inf_{\mathbf{v}\in\mathbb{R}^{N_{{\scalebox{.5}{$\ell$}}}}\setminus \{0 \}} \frac{\mathbf{v}^T \widehat{\mathbf{Y}} \mathbf{v}}{\mathbf{v}^T\mathbf{v}} \leq \sup_{\mathbf{v}\in\mathbb{R}^{N_{{\scalebox{.5}{$\ell$}}}}\setminus \{0 \}} \frac{\mathbf{v}^T \widehat{\mathbf{Y}} \mathbf{v}}{\mathbf{v}^T\mathbf{v}} \leq \hCsup < \infty
\end{equation}
and
\begin{equation}\label{eq:spechatY}
0 < \hCinf \leq \inf_{\mathbf{v}\in\mathbb{R}^{N_{{\scalebox{.5}{$\ell$}}}}\setminus \{0 \}} \frac{\mathbf{v}^T \mathbf{Y} \mathbf{v}}{\mathbf{v}^T\mathbf{v}} \leq \sup_{\mathbf{v}\in\mathbb{R}^{N_{{\scalebox{.5}{$\ell$}}}}\setminus \{0 \}} \frac{\mathbf{v}^T \mathbf{Y} \mathbf{v}}{\mathbf{v}^T\mathbf{v}} \leq \hCsup < \infty.
\end{equation}
For a discussion on the scaling of the values $\hCinf,\,\hCsup$ and the assumptions on $\theta$ in terms of the mesh sizes $H$ and $h$, we refer to the proof of Corollary~\ref{cor:network} below.
When the matrices $\mathbf{Y}$ and $\widehat{\mathbf{Y}}$ are rescaled with $\hcresc\in(0,\hCsup^{-1})$, it holds that
\begin{equation*}
\| \mathrm{\mathbf{Id}}_{N_{{\scalebox{.5}{$\ell$}}}^2} - \hcresc\mathbf{Y}\|_2 \leq 1-\widehat{\delta},
\quad \mathrm{as\, well\, as} \quad
\| \mathrm{\mathbf{Id}}_{N_{{\scalebox{.5}{$\ell$}}}^2} - \hcresc\widehat{\mathbf{Y}}\|_2 \leq 1-\widehat{\delta},
\end{equation*}
with $\widehat{\delta} : = \hcresc \hCinf$. The rescaling is implemented by the network $\big((\hcresc\, \mathrm{\mathbf{Id}}_{N_{{\scalebox{.5}{$\ell$}}}^2}, \mathbf{0}_{N_{{\scalebox{.5}{$\ell$}}}^2})\big).$ Moreover, changing this network to $\big((-\hcresc\, \mathrm{\mathbf{Id}}_{N_{{\scalebox{.5}{$\ell$}}}^2}, \mathrm{vec}(\mathrm{\mathbf{Id}}_{N_{{\scalebox{.5}{$\ell$}}}^2}))\big)$ by switching signs of the weight matrix and adding a vectorized identity matrix as a bias term leads to the implementation of $(\mathrm{\mathbf{Id}}_{N_{{\scalebox{.5}{$\ell$}}}^2} - \hcresc\widehat{\mathbf{Y}})$ if $\mathrm{vec}(\widehat{\mathbf{Y}})$ is fed through this network. Theorem \ref{theo:invnn} yields the existence of a neural network $\Psi_{\mathrm{inv}, \gamma}^{1-\widehat{\delta}, N_{\ell}^2}$ such that 
\begin{equation}\label{eq:step5}
\Big\|\, \big(\hcresc\widehat{\mathbf{Y}}\big)^{-1} - \mathrm{mat}\Big(R_\rho^{K_{N_{{\scalebox{.5}{$\ell$}}}}^{1-\widehat{\delta}} } \big(\Psi_{\mathrm{inv},\gamma}^{1-\widehat{\delta}, N_{{\scalebox{.5}{$\ell$}}}^2 }\big)(\mathrm{vec}(\mathrm{\mathbf{Id}}_{N_{{\scalebox{.5}{$\ell$}}}^2} - \hcresc\widehat{\mathbf{Y}})) \Big)\, \Big\|_2 \leq \gamma.
\end{equation}
Note that this in turn implies that
\begin{equation}\label{eq:step5.1}
\Big\|\, \left(\hcresc\mathbf{Y}\right)^{-1} - \mathrm{mat}\Big(R_\rho^{K_{N_{{\scalebox{.5}{$\ell$}}}}^{1-\widehat{\delta}} } \big(\Psi_{\mathrm{inv},\eta}^{1-\widehat{\delta}, N_{{\scalebox{.5}{$\ell$}}}^2 }\big)\big(\mathrm{vec}(\mathrm{\mathbf{Id}}_{N_{{\scalebox{.5}{$\ell$}}}^2} - \hcresc\widehat{\mathbf{Y}})\big) \Big)\, \Big\|_2 
\leq \frac{\cresc\norm{\mathbf{I}_\omega}_2^2}{\hcresc\hCinf^2}\, \theta + \gamma.
\end{equation}
Indeed, it holds by the triangle inequality and \eqref{eq:step5} that
\begin{align*}
\Big\|\, \big(\hcresc&\mathbf{Y}\big)^{-1} - \mathrm{mat}\Big(R_\rho^{K_{N_{{\scalebox{.5}{$\ell$}}}}^{1-\widehat{\delta}} } \big(\Psi_{\mathrm{inv},\eta}^{1-\widehat{\delta}, N_{{\scalebox{.5}{$\ell$}}}^2 }\big)(\mathrm{vec}(\mathrm{\mathbf{Id}}_{N_{{\scalebox{.5}{$\ell$}}}^2} - \hcresc\widehat{\mathbf{Y}})) \Big)\, \Big\|_2  \\
\hspace*{1cm}&\leq \Big\|\, \big(\hcresc\mathbf{Y}\big)^{-1} - \big(\hcresc \mathbf{\widehat{Y}}\big)^{-1}  \Big\|_2 \\
&+ \Big\|\, \big(\hcresc\mathbf{\widehat{Y}}\big)^{-1} - \mathrm{mat}\Big(R_\rho^{K_{N_{{\scalebox{.5}{$\ell$}}}}^{1-\widehat{\delta}} } \big(\Psi_{\mathrm{inv},\eta}^{1-\widehat{\delta}, N_{{\scalebox{.5}{$\ell$}}}^2 }\big)(\mathrm{vec}(\mathrm{\mathbf{Id}}_{N_{{\scalebox{.5}{$\ell$}}}^2} - \hcresc\widehat{\mathbf{Y}})) \Big)\, \Big\|_2  =: (\star) + \gamma.
\end{align*}
Using $\norm{\mathbf{A}^{-1}-\mathbf{B}^{-1}}_2 \leq \norm{\mathbf{A}-\mathbf{B}}_2 \norm{\mathbf{A}^{-1}}_2 \norm{\mathbf{B}^{-1}}_2$, the term $(\star)$ can be estimated by
\begin{equation*}
\begin{aligned}
(\star) &\leq \left\|\, \hcresc\mathbf{Y} - \hcresc \mathbf{\widehat{Y}}  \right\|_2
\left\|\, \big(\hcresc\mathbf{Y}\big)^{-1} \right\|_2
\left\|\, \big(\hcresc \mathbf{\widehat{Y}}\big)^{-1}  \right\|_2 
\leq \frac{1}{\hcresc} \norm{\mathbf{I}_{\omega}}_2^2 \cresc\, \theta \, \frac{1}{\hCinf} \, \frac{1}{\hCinf},
\end{aligned}
\end{equation*}
where we have used \eqref{eq:step2} and the definitions as well as the spectral bounds given on $\mathbf{Y},\widehat{\mathbf{Y}}$ in \eqref{eq:specY}, \eqref{eq:spechatY}. This yields the inequality \eqref{eq:step5.1}.

Rescaling the output of $\Psi_{\mathrm{inv}, \gamma}^{1-\widehat{\delta}, N_{\ell}^2}$ again, we get a network with the desired approximation properties defined by
\begin{equation*}
\Psi^5 := \big((\hcresc\, \mathrm{\mathbf{Id}}_{N_{{\scalebox{.5}{$\ell$}}}^2}, \mathbf{0}_{N_{{\scalebox{.5}{$\ell$}}}^2})\big) \bullet \Psi_{\mathrm{inv}, \gamma}^{1-\widehat{\delta}, N_{{\scalebox{.5}{$\ell$}}}^2} \odot \big((-\hcresc\, \mathrm{\mathbf{Id}}_{N_{{\scalebox{.5}{$\ell$}}}^2},\mathrm{vec}(\mathrm{\mathbf{Id}}_{N_{{\scalebox{.5}{$\ell$}}}^2}))\big).
\end{equation*}
Further, we have the following bounds on depth and number of parameters, 
\begin{itemize}
\item[(i)] $L(\Psi^5) \lesssim \log(m(\gamma,\widehat{\delta})) \, \left(\log(1/\gamma) +  \log(m(\gamma,\widehat{\delta})) + \log(N_{\ell})\right) + 1, $
\item[(ii)] $M(\Psi^5) \lesssim m(\gamma,\widehat{\delta}) \log^2(m(\gamma,\widehat{\delta}))N_{\ell}^3 \, \left(\log(1/\gamma) +  \log(m(\gamma,\widehat{\delta})) + \log(N_{\ell})\right) + N_{\ell}^2.$ \\
\end{itemize}

\emph{Step 6:} The multiplication of $\mathbf{Y}^{-1},$ respectively $\widehat{\mathbf{Y}}^{-1}$, with $\mathbf{I}_\omega\mathbf{P}_{\omega,K}$ from the right is analogous to Step 3. Again, we have that both $\mathbf{Y}^{-1}$ and $\widehat{\mathbf{Y}}^{-1}$ are symmetric and thus the expression $\mathbf{Y}^{-1} \mathbf{I}_\omega\mathbf{P}_{\omega, K}$ can also be written as $((\mathbf{I}_\omega\mathbf{P}_{\omega, K})^T \mathbf{Y}^{-1})^T.$ With the same argument as in Step 3, the map $\mathrm{vec}(\mathbf{Y}^{-1}) \mapsto \mathrm{vec}(\mathbf{Y}^{-1}\mathbf{I}_\omega\mathbf{P}_{\omega,K})$ is thus exactly implemented by 
\begin{equation*}
\Psi^6 := ((\widetilde{\mathbf{Q}}, \mathbf{0}_{\, 2^d N_{{\scalebox{.5}{$\ell$}}}})) \bullet P(\Psi_1^6, \dots ,\Psi^6_{N_{{\scalebox{.5}{$\ell$}}}}),
\end{equation*}
where $\widetilde{\mathbf{Q}}$ is a suitable permutation matrix of dimension $ N_{\ell}2^d \times  N_{\ell} 2^d$ and $\Psi_1^6, \dots ,\Psi^6_{N_{{\scalebox{.5}{$\ell$}}}}$ are identical copies of 
\begin{equation*}
\Psi_{(\mathbf{I}_\omega\mathbf{P}_{\omega, K})^T} := \big((\mathbf{I}_\omega\mathbf{P}_{\omega, K})^T, \mathbf{0}_{2^d}) \big).
\end{equation*}
Since $\mathbf{I_\omega}$ is a quasi-interpolation operator on $\omega$ and $\mathbf{P}_{\omega, K}$ is the prolongation operator from the element $K$ to the whole patch, we roughly estimate $\|\mathbf{I_\omega}\mathbf{P}_{\omega, K} \|_0 \leq 6^d.$  Combining this with Lemma \ref{lem:parann} yields
\begin{itemize}
\item[(i)] $L(\Psi^6) = 1,$
\item[(ii)] $M(\Psi^6) = 6^d N_{\ell} = 6^d (2\ell+2)^d \lesssim \ell^d.$ \\
\end{itemize}

\emph{Step 7:} Analogous to Step 4, we consider the parallelization 
\begin{equation*}
\Psi^7 := P(\Psi_1^7, \dots ,\Psi^7_{N_{{\scalebox{.5}{$\ell$}}}}),
\end{equation*}
where $\Psi_1^6, \dots ,\Psi^6_{N_{{\scalebox{.5}{$\ell$}}}}$ are identical copies of 
\begin{equation*}
\Psi_{\mathbf{P}_{\omega}^T \mathbf{I}_\omega^T} := \big((\mathbf{P}_{\omega}^T \mathbf{I}_\omega^T, \mathbf{0}_{2^d}) \big).
\end{equation*}
Again, no permutation is necessary in this step. With our choice of the quasi-interpolation operator, the rough estimate $\|\mathbf{P}_\omega^T \mathbf{I} _\omega^T \|_0 \leq 3^d N_{\ell}$ holds, which leads to 
\begin{itemize}
\item[(i)] $L(\Psi^7) = 1,$
\item[(ii)] $M(\Psi^7) \leq 3^dN_{\ell}^2 = 3^d (2\ell+2)^{2d}\lesssim \ell^{2d}.$ \\
\end{itemize}

By connecting the individual networks $\Psi^1, \dots, \Psi^7$ in series by sparsely concatenating them, i.e., defining 
\begin{equation*}
\Psi^{\mathrm{pg}}_\eta := \Psi^7 \odot \dots \odot \Psi^1,
\end{equation*}
we finally obtain a network with the desired properties. In particular, we have
\begin{equation*}
\norm{\mathbf{S}_{\omega}^\mathrm{pg} - \mathrm{mat}(\calR^{\mathfrak{A}_{\varepsilon,\omega}}_\rho(\Psi^{\mathrm{pg}}_\eta(\mathbf{A}_{\varepsilon,\omega}))}_2 \leq \frac{\cresc\norm{\mathbf{P}_{\omega,K}}_2\norm{\mathbf{P}_{\omega}}_2\norm{\mathbf{I}_{\omega}}_2^4}{\hCinf^2}\, \theta + \norm{\mathbf{P}_{\omega,K}}_2\norm{\mathbf{P}_{\omega}}_2\norm{\mathbf{I}_{\omega}}_2^2 \gamma, 
\end{equation*}
since with the abbreviation
\begin{equation*}
\widehat{\mathrm{inv}}(\hcresc\widehat{\mathbf{Y}}) := 
\mathrm{mat}\Big(R_\rho^{K_{N_{{\scalebox{.5}{$\ell$}}}}^{1-\widehat{\delta}} } \big(\Psi_{\mathrm{inv},\eta}^{1-\widehat{\delta}, N_{\ell}^2 }\big)(\mathrm{vec}(\mathrm{\mathbf{Id}}_{N_{{\scalebox{.5}{$\ell$}}}^2} - \hcresc\widehat{\mathbf{Y}})) \Big),
\end{equation*}
it holds that
\begin{equation*}
\begin{aligned}
\Big\| \mathbf{S}_{\omega}^\mathrm{pg} &- \mathrm{mat}(\calR^{\mathfrak{A}_{\varepsilon,\omega}}_\rho(\Psi^{\mathrm{pg}}_\eta(\mathbf{A}_{\varepsilon,\omega}))\Big\|_2  \\ \hspace{1cm}
&= \left\|\, \mathbf{I}_\omega \mathbf{P}_\omega \hcresc\, \left(\hcresc\mathbf{Y}\right)^{-1} \mathbf{I}_\omega^T \mathbf{P}_{K,\omega} - \mathbf{I}_\omega \mathbf{P}_\omega \hcresc \widehat{\mathrm{inv}}(\hcresc\widehat{\mathbf{Y}}) \mathbf{I}_\omega^T \mathbf{P}_{K,\omega}\, \right\|_2
\\
&\leq \norm{\mathbf{P_\omega}}_2 \norm{\mathbf{P}_{K,\omega}}_2 \norm{\mathbf{I_\omega}}_2^2 \,\hcresc  \Big(\frac{\cresc\norm{\mathbf{I}_\omega}_2^2}{\hcresc\hCinf^2}\, \theta + \gamma\Big)
\end{aligned}
\end{equation*}
due to \eqref{eq:step5.1}. Choosing $\theta,\gamma$ such that 
\begin{equation*}
\frac{\cresc\norm{\mathbf{P}_{\omega,K}}_2\norm{\mathbf{P}_{\omega}}_2\norm{\mathbf{I}_{\omega}}_2^4}{\hCinf^2}\, \theta + \norm{\mathbf{P}_{\omega,K}}_2\norm{\mathbf{P}_{\omega}}_2\norm{\mathbf{I}_{\omega}}_2^2 \gamma \leq \eta
\end{equation*}
then yields the estimate. As above, we refer to the proof of Corollary \ref{cor:network} below for a discussion of the scaling of the relevant spectral norms in terms of the mesh sizes $H,h$. Using Lemma \ref{lem:sparseccnn}, we obtain the following bounds for the final network, 
\begin{itemize}
\item[(i)] $L(\Psi^\mathrm{pg}) \leq \sum_{i=1}^7 L(\Psi^i) \lesssim \log(m(\theta,\delta)) \, \big(\log(1/\theta) +  \log(m(\theta,\delta)) + \log(n_{\ell})\big) \\
\hspace*{4cm}+ \log(m(\gamma,\widehat{\delta})) \, \big(\log(1/\gamma) +  \log(m(\gamma,\widehat{\delta})) + \log(N_{\ell})\big) + 1 $,
\item[(ii)] $M(\Psi^\mathrm{pg}) \lesssim \sum_{i=1}^7 M(\Psi^i) \\
\hspace*{1.35cm} \lesssim m(\theta,\delta) \log^2(m(\theta,\delta))n_{\ell}^3 \, \big(\log(1/\theta) +  \log(m(\theta,\delta)) + \log(n_{\ell})\big)
\\ \hspace*{1.3cm} + m(\gamma,\widehat{\delta}) \log^2(m(\gamma,\widehat{\delta}))N_{\ell}^3 \big(\log(1/\gamma) +  \log(m(\gamma,\widehat{\delta})) + \log(N_{\ell})\big) + (\ell H/h)^{2d}$.
\end{itemize}
This is the assertion. \qed

\subsection{Proof of Corollary \ref{cor:network}}\label{app:network}
The claim follows directly from Theorem~\ref{theo:lodnn} with $\eta \approx H^k$ and the estimation of the quantities $\theta, \gamma, \delta, \widehat{\delta}$ in terms of the mesh sizes $H, h$. 
Starting with $\theta$ and $\gamma,$ we have the condition that
\begin{equation}\label{eq:boundTheta}
\theta<\min\Bigg\{H^k, \lambda_{\mathrm{min}}(\mathbf{I}_\omega \mathbf{S}^{-1}\mathbf{I}_\omega^T), \frac{\lambda_{\mathrm{min}}(\mathbf{I}_\omega \mathbf{S}^{-1}\mathbf{I}_\omega^T)}{2\,\cresc\|\mathbf{I}_\omega \|_2^2}\Bigg\},
\end{equation}
which requires an estimation of the smallest eigenvalue of $\mathbf{I}_\omega \mathbf{S}^{-1}\mathbf{I}_\omega^T.$ Observe that 
\begin{equation*}
\lambda_\mathrm{min}(\mathbf{I}_\omega \mathbf{S}^{-1}\mathbf{I}_\omega^T) \geq \lambda_\mathrm{min}(\mathbf{S}^{-1}) \lambda_\mathrm{min}(\mathbf{I}_\omega \mathbf{I}_\omega^T) = \frac{\lambda_\mathrm{min}(\mathbf{I}_\omega \mathbf{I}_\omega^T)}{\lambda_\mathrm{max}(\mathbf{S})} \gtrsim \frac{\lambda_\mathrm{min}(\mathbf{I}_\omega \mathbf{I}_\omega^T)}{h^{d-2}},
\end{equation*}
where we have used the well-known fact that $\lambda_\mathrm{max}(\mathbf{S}) \approx h^{d-2}$ as the maximal eigenvalue of a finite element stiffness matrix. In order to estimate the eigenvalues of $\mathbf{I}_\omega\mathbf{I}_\omega^T,$ we have to look into the practical computation of the matrix $\mathbf{I}_\omega$. It holds that 
\begin{equation*}
\mathbf{I}_\omega = \mathbf{R}_{\mathrm{c},\omega} \mathbf{E}_{\mathrm{c},\omega} (\mathbf{M}_{\mathrm{c},\omega}^\mathrm{dg})^{-1} (\mathbf{P}_\omega^\mathrm{dg})^T \mathbf{M}_\omega^\mathrm{dg} \mathbf{C}_\omega \mathbf{R}_\omega^T, 
\end{equation*}
where $\mathbf{E}_{\mathrm{c},\omega}$ is the algebraic realization of the averaging operator $E_H$ introduced in~\eqref{eq:IH}, $\mathbf{M}_{\mathrm{c},\omega}^\mathrm{dg}$ and $\mathbf{M}_\omega^\mathrm{dg}$ are the mass matrices corresponding to $Q1$ discontinuous finite element functions on $\mathcal{T}_H(\omega)$ and $\mathcal{T}_h(\omega)$, respectively, $\mathbf{P}_\omega^\mathrm{dg}$ is the prolongation map from $\mathcal{T}_H(\omega)$ to $\mathcal{T}_h(\omega)$ for $Q1$ discontinuous finite element functions, $\mathbf{C}_\omega$ maps the vector representation of a continuous finite element function to the vector of its discontinuous representation, and $\mathbf{R}_\omega$ is a restriction operator that removes all entries corresponding to boundary nodes of the patch $\omega$ on the fine mesh. The restriction operator $\mathbf{R}_{\mathrm{c},\omega}$ removes coarse boundary nodes on $\partial D$ only. Since $\mathbf{I}_\omega$ has full rank, it holds that
\begin{equation*}
\begin{aligned}
\lambda_\mathrm{min}(\mathbf{I}_\omega \mathbf{I}_\omega^T) \geq&\,  \lambda_\mathrm{min}(\mathbf{R}_{\mathrm{c},\omega} \mathbf{R}_{\mathrm{c},\omega}^T)\, \lambda_\mathrm{min}(\mathbf{E}_{\mathrm{c},\omega} \mathbf{E}_{\mathrm{c},\omega}^T)\, \lambda_\mathrm{min}((\mathbf{M}_{\mathrm{c},\omega}^\mathrm{dg})^{-2})\, \lambda_\mathrm{min}((\mathbf{P}_\omega^\mathrm{dg})^T \mathbf{P}_\omega^\mathrm{dg})\, \\  &\cdot\lambda_\mathrm{min}((\mathbf{M}_\omega^\mathrm{dg})^2)\, \lambda_\mathrm{min}(\mathbf{C}_\omega^T \mathbf{C}_\omega)\,  \lambda_\mathrm{min}(\mathbf{R}_\omega \mathbf{R}_\omega^T).
\end{aligned}
\end{equation*}
Using that $\lambda_\mathrm{min}(\mathbf{R}_{\mathrm{c},\omega} \mathbf{R}_{\mathrm{c},\omega}^T) = \lambda_\mathrm{min}(\mathbf{R}_\omega \mathbf{R}_\omega^T) = 1,\, \lambda_\mathrm{min}(\mathbf{E}_{\mathrm{c},\omega} \mathbf{E}_{\mathrm{c},\omega}^T) \gtrsim 2^{-d},\, \lambda_\mathrm{min}((\mathbf{M}_{\mathrm{c},\omega}^\mathrm{dg})^{-2}) \gtrsim H^{-2d},\, \lambda_\mathrm{min}((\mathbf{M}_\omega^\mathrm{dg})^2) \gtrsim h^{2d},\, \lambda_\mathrm{min}((\mathbf{P}_\omega^\mathrm{dg})^T \mathbf{P}_\omega^\mathrm{dg})\gtrsim (H/h)^d$ and $\lambda_\mathrm{min}(\mathbf{C}_\omega^T \mathbf{C}_\omega)=1,$ we obtain
\begin{equation*}
\lambda_\mathrm{min}(\mathbf{I}_\omega \mathbf{I}_\omega^T) \gtrsim 2^{-d} H^{-2d} h^{2d} (H/h)^d \gtrsim (h/H)^{d}. 
\end{equation*}
This, in turn, yields
\begin{equation*}
\lambda_\mathrm{min}(\mathbf{I}_\omega \mathbf{S}^{-1}\mathbf{I}_\omega^T) \gtrsim \frac{(h/H)^{d}}{h^{d-2}} \approx \frac{h^{2}}{H^{d}} 
\end{equation*}
With a similar argument for the maximal eigenvalue, we also get $\|\mathbf{I}_\omega\|_2^2 \lesssim (h/H)^{d}$. 
Further, we have
\begin{equation*}
\cresc < \frac{1}{\Csup} \leq  \frac{1}{\lambda_\mathrm{max}(\mathbf{S})} \approx h^{2-d}. 
\end{equation*}
Combining the above estimates, we arrive at
\begin{equation*}
\frac{\lambda_{\mathrm{min}}(\mathbf{I}_\omega \mathbf{S}^{-1}\mathbf{I}_\omega^T)}{\cresc\|\mathbf{I}_\omega \|_2^2} \gtrsim \frac{h^{2}/H^{d}}{h^{2-d}(h/H)^{d}} \approx 1.
\end{equation*}
Recall that according to Theorem~\ref{theo:lodnn}, the parameter $\gamma$ has to satisfy the condition
\begin{equation*}
\frac{\cresc\norm{\mathbf{P}_{\omega,K}}_2\norm{\mathbf{P}_{\omega}}_2\norm{\mathbf{I}_{\omega}}_2^4}{\hCinf}\, \theta + \norm{\mathbf{P}_{\omega,K}}_2\norm{\mathbf{P}_{\omega}}_2\norm{\mathbf{I}_{\omega}}_2^2 \gamma \leq H^k.
\end{equation*}
Since $\mathbf{P}_{\omega}$ and $\mathbf{P}_{\omega,K}$ realize prolongations, we have $\|\mathbf{P}_\omega\|_2 \lesssim (H/h)^{d/2}$ as well as $\|\mathbf{P}_{\omega,K}\|_2 \lesssim (H/h)^{d/2}$. With a suitable choice of $\theta$, i.e.
\begin{equation*}
\theta \approx \frac{h^{2}}{H} H^k \Bigg(1 + \frac{H^d}{h^2}\Bigg)^{-1},
\end{equation*}
the estimate~\eqref{eq:boundTheta} is fulfilled. Further, we have with~\eqref{eq:eigBound1} that
\begin{equation*}
\hCinf \approx \lambda_\mathrm{min}(\mathbf{I}_\omega \mathbf{S}^{-1}\mathbf{I}_\omega^T) \gtrsim \frac{h^{2}}{H^d}.
\end{equation*}
Choosing $\gamma \approx \theta$ thus leads to the rough estimate 
\begin{equation*}
\begin{aligned}
\frac{\cresc\norm{\mathbf{P}_{\omega,K}}_2\norm{\mathbf{P}_{\omega}}_2\norm{\mathbf{I}_{\omega}}_2^4}{\hCinf^2}\, &\theta + \norm{\mathbf{P}_{\omega,K}}_2\norm{\mathbf{P}_{\omega}}_2\norm{\mathbf{I}_{\omega}}_2^2 \gamma \\
& \lesssim \frac{h^{2-d} (H/h)^d (h/H)^{2d}}{(h^{4}/H^{2d})}\theta + (H/h)^{d} (h/H)^{d} \gamma \\
& \approx \Big(1 + \frac{H^d}{h^2}\Big)\, \theta 
\lesssim \Big(1 + \frac{H^d}{h^2}\Big)\, \Big(1 + \frac{H^d}{h^2}\Big)^{-1}H^k = H^k.
\end{aligned}
\end{equation*}
In order to estimate the quantities $m(\theta,\delta)$ and $m(\gamma,\widehat{\delta}),$ we derive lower bounds on $\delta,\, \widehat{\delta}$, and $\theta$.
If $\cresc$ is chosen close to $\Csup^{-1}\,$, we have that 
\begin{equation*}
\delta = \cresc \Cinf \approx \frac{\Cinf}{\Csup} \approx \frac{\lambda_\mathrm{min}(\mathbf{S})}{\lambda_\mathrm{max}(\mathbf{S})} \approx \frac{h^d}{h^{d-2}} = h^{2},
\end{equation*}
using the fact that $\lambda_\mathrm{min}(\mathbf{S}) \approx h^d$. Analogously, for a choice of $\hcresc$ close to $\hCsup^{-1}$ we obtain 
\begin{equation*}
\widehat{\delta} \approx \frac{\hCinf}{\hCsup} \approx \frac{\lambda_\mathrm{min}(\mathbf{I}_\omega \mathbf{S}^{-1}\mathbf{I}_\omega^T)}{\lambda_\mathrm{max}(\mathbf{I}_\omega \mathbf{S}^{-1}\mathbf{I}_\omega^T)} \gtrsim \frac{{h^{2}}{H^{-d}}}{H^{-d}} \approx h^2,
\end{equation*}
where we have used that $\lambda_\mathrm{max}(\mathbf{I}_\omega \mathbf{S}^{-1} \mathbf{I}_\omega^T) \lesssim H^{-d}$, which can be derived analogously to the scaling of the smallest eigenvalue. For $\theta$ (and thus $\gamma$), a very rough estimation yields 
\begin{equation*}
\begin{aligned}
\theta &\approx \frac{h^{2}}{H} H^k \Bigg(1 + \frac{H^d}{h^2}\Bigg)^{-1} 
\approx H^{k-1}h^2 \frac{h^2}{h^2 +H^d} 
\gtrsim H^{k-1}h^4.
\end{aligned}
\end{equation*}
Therefore, we obtain
\begin{equation*}
m(\theta,\delta) = \left\lceil \frac{\log(0.5 \theta\delta)}{\log(1-\delta)} \right\rceil \lesssim \frac{\log(0.5 H^{k-1}h^6)}{\log(1-h^2)}\lesssim \frac{(k-1)|\log(H)| + 6|\log(h)|}{h^2}
\end{equation*}
and similarly
\begin{equation*}
m(\gamma,\widehat{\delta}) \lesssim \frac{(k-1)|\log(H)| + 6|\log(h)|}{h^2}.
\end{equation*}
Moreover, it holds $N_\ell < n_\ell \lesssim (\vert \log(H)\vert H/h)^d$ and thus
\begin{equation*} 
\log(n_\ell) \lesssim d\log(H\vert\log(H)\vert/h) \lesssim |\log(h)|. 
\end{equation*}
Inserting all these results into the bounds derived in Theorem~\ref{theo:invnn} and using without loss of generality that $\theta \leq \gamma$ and $\delta \leq \widehat \delta$, we obtain 
\begin{equation*}
\begin{aligned}
L(\Psi^\mathrm{pg}_{H^k}) &\lesssim \log(m(\theta,\delta))\Big(\log(1/\theta) + \log(m(\theta,\delta)) + \log(n_\ell) \Big) \\
&\lesssim \Big(|\log(h)| + |\log(k)|\Big) \Big(|\log(h)| + |\log(k)|  \Big) \\
&\lesssim |\log(h)|^2 + |\log(k)|^2
\end{aligned}
\end{equation*}
as well as
\begin{equation*}
\begin{aligned}
M(\Psi^\mathrm{pg}_{H^k}) &\lesssim m(\theta,\delta) \log^2(m(\theta,\delta))\,n_\ell^3\, \Big(\log(1/\theta) + \log(m(\theta,\delta)))+ \log(n_\ell) \Big) 
\\&\qquad+ (\vert \log(H)\vert H/h)^{2d}\\
&\lesssim h^{-2}\Big(k |\log(H)| + |\log(h)|\Big)\Big(|\log(h)|^3 + |\log(k)|^3\Big) \big(\vert \log(H)\vert H/h\big)^{3d}. 
\end{aligned}
\end{equation*}
This is the assertion. \qed

\section{Proofs in Section \ref{sec: diffSol}} \label{app:diffSol}
\subsection{Proof of Lemma \ref{lem:CGvsPG}} \label{app:CGvsPG}
Taking the difference of~\eqref{eq:discvarproblemclassic} and~\eqref{eq:discvarproblempg} leads to
\begin{equation*}
a((\mathsf{id}-\calQ^\ell)(\uCG-\uPG),(\mathsf{id}-\calQ^\ell)v_H) = a((\mathsf{id}-\calQ^\ell)\uPG,\calQ^\ell v_H) 
\end{equation*}
for all $v_H \in \VH$. Therefore, with $v_H = (\uCG-\uPG)$, we have
\begin{align*}
a((\mathsf{id}-\calQ^\ell&)(\uCG-\uPG),(\mathsf{id}-\calQ^\ell)(\uCG-\uPG)) \\& = a((\mathsf{id}-\calQ^\ell)\uPG,\calQ^\ell(\uCG-\uPG)) \\
& = \underbrace{a((\mathsf{id}-\calQ)\uPG,\calQ^\ell(\uCG-\uPG))}_{=\,0} + a((\calQ-\calQ^\ell)\uPG,\calQ^\ell(\uCG-\uPG))\\
& \lesssim \exp(-c\ell)\|\nabla \uPG\|_{L^2(D)}\,\|(\mathsf{id}-\calQ^\ell)(\uCG-\uPG)\|_{L^2(D)}
\end{align*}
using~\eqref{eq:loc} in the last step. 
This leads to
\begin{equation*}
\|\nabla (\mathsf{id}-\calQ^\ell)(\uCG-\uPG)\|_{L^2(D)} \lesssim \exp(-c\ell)\|\nabla \uPG\|_{L^2(D)}
\end{equation*}
and, with the Friedrichs inequality, interpolation estimates, and the stability of the solution~$\uPG$, we get
\begin{equation*}
\|\uCG- \uPG\|_{L^2(D)} \lesssim \exp(-c\ell).
\end{equation*}
Further, we have 
\begin{equation*}
\|\mathbf{u}^\mathrm{c} - \textbf{u}^\mathrm{pg}\|_2 \leq C H^{-d/2}\|\uCG- \uPG\|_{L^2(D)} \lesssim H^{-d/2}\exp(-c\ell),
\end{equation*}
which is the assertion. \qed

\subsection{Proof of Lemma~\ref{lem:boundR}}\label{app:boundR}
Due to the definition of~\eqref{eq:discvarproblemclassic} (which defines $\mathbf{S}^\mathrm{c}$), we have for any vector $\mathbf{v}$ with corresponding function $v\in \VH$ that
\begin{equation}\label{eq:proofEV1}
\begin{aligned}
\mathbf{v}^T \mathbf{S}^\mathrm{c} \mathbf{v} &= a((\mathsf{id}-\calQ^\ell)v,(\mathsf{id}-\calQ^\ell)v) \geq \alpha \|\nabla(\mathsf{id}-\calQ^\ell)v\|_{L^2(D)}^2 \\
&\geq \alpha C^{-2}\|\nabla v\|_{L^2(D)}^2 
\end{aligned}
\end{equation}
using the interpolation estimate~\eqref{eq:IHprop} and $\IH(\mathsf{id}-\calQ^\ell)v = v$. With the Poincar\'e--Friedrichs inequality with constant $C_\mathrm{P}$ and the estimate~\eqref{eq:normEquiv}, we further get
\begin{equation}\label{eq:proofEV2} 
\mathbf{v}^T \mathbf{S}^\mathrm{c} \mathbf{v}
\geq \alpha C^{-2}C_\mathrm{P}^{-2}\|v\|_{L^2(D)}^2 
= \alpha C^{-2}C_\mathrm{P}^{-2}\mathbf{v}^T\mathbf{M}\mathbf{v} \geq cH^{d} \|\mathbf{v}\|_2^2.
\end{equation}
The claim follows since the minimal eigenvalue is bounded from above by the Rayleigh quotient. \qed

\subsection{Proof of Theorem \ref{th:error}}\label{app:error}
We only show the first estimate. The second one follows directly from the equivalence 
\begin{equation}\label{eq:equiv}
H^{d/2} \|\mathbf{u}^\mathrm{pg} - \mathbf{u}^\mathrm{nn}\|_2 \lesssim \|\uPG - \uNN\|_{L^2(D)} \lesssim H^{d/2} \|\mathbf{u}^\mathrm{pg} - \mathbf{u}^\mathrm{nn}\|_2,
\end{equation}
which follows from~\eqref{eq:normEquiv}.
With Lemma~\ref{lem:CGvsPG}, we have
\begin{equation}\label{eq:proof1}
\| \mathbf{u}^\mathrm{pg} - \mathbf{u}^\mathrm{nn}\|_2 \leq \| \mathbf{u}^\mathrm{pg} - \mathbf{u}^\mathrm{c}\|_2 + \| \mathbf{u}^\mathrm{c} - \mathbf{u}^\mathrm{nn}\|_2 \lesssim H^{-d/2}e^{-c_\mathrm{loc}\ell} + \| \mathbf{u}^\mathrm{c} - \mathbf{u}^\mathrm{nn}\|_2.
\end{equation}
We now bound the second term. 
Note that the right-hand side involving $f$ is deterministic and thus equal for the discrete problems~\eqref{eq:discvarproblempg}, \eqref{eq:discvarproblemclassic}, and the system based on the neural network. Denoting the vector-version of the right-hand side with~$\mathbf{f}$, we particularly have the linear systems
\begin{equation}\label{eq:linSys}
\mathbf{S}^\mathrm{pg} \mathbf{u}^\mathrm{pg} = \mathbf{f},\qquad \mathbf{S}^\mathrm{c} \mathbf{u}^\mathrm{c} = \mathbf{f},\qquad \mathbf{S}^\mathrm{nn} \mathbf{u}^\mathrm{nn} = \mathbf{f},
\end{equation} 
where $\mathbf{S}^\mathrm{pg}$ is the PG-LOD stiffness matrix, $\mathbf{S}^\mathrm{c}$ the (symmetric) C-LOD stiffness matrix, and $\mathbf{S}^\mathrm{nn}$ the one assembled from the neural network.
Using~\eqref{eq:linSys}, we estimate
\begin{equation}\label{eq:proof2}
\begin{aligned}
\| \mathbf{u}^\mathrm{c} - \mathbf{u}^\mathrm{nn}\|_2 &= \big\| \big(\mathbf{S}^\mathrm{c}\big)^{-1}\mathbf{f} - \mathbf{u}^\mathrm{nn}\big\|_2 = \big\| \big(\mathbf{S}^\mathrm{c}\big)^{-1}(\mathbf{S}^\mathrm{nn}\mathbf{u}^\mathrm{nn}  - \mathbf{S}^\mathrm{c}\mathbf{u}^\mathrm{nn})\big\|_2\\
& \leq \big\| \big(\mathbf{S}^\mathrm{c}\big)^{-1}\big\|_2\, \big\|\mathbf{S}^\mathrm{nn}  - \mathbf{S}^\mathrm{c}\big\|_2\, \big\|\mathbf{u}^\mathrm{nn}\big\|_2.
\end{aligned}
\end{equation}
With the eigenvalue bound on $\mathbf{S}^\mathrm{c}$ that is proven in Lemma~\ref{lem:boundR}, we obtain 
\begin{equation}\label{eq:invMatrix}
\big\|\big(\mathbf{S}^\mathrm{c}\big)^{-1}\big\|_2 = \big(\lambda_\mathrm{min}(\mathbf{S}^\mathrm{c})\big)^{-1} \lesssim H^{-d}. 
\end{equation}
For the second factor on the right-hand side of~\eqref{eq:proof2}, we estimate
\begin{align*}
\|\mathbf{S}^\mathrm{nn}  - \mathbf{S}^\mathrm{c}\|_2 \leq \|\mathbf{S}^\mathrm{nn}  - \mathbf{S}^\mathrm{pg}\|_2 + \|\mathbf{S}^\mathrm{pg}  - \mathbf{S}^\mathrm{c}\|_2.
\end{align*}
With the choice $k = 2d + 1$ in Corollary~\ref{cor:network} and~\eqref{eq:decS}-\eqref{eq:Theta}, we obtain the bounds in~\eqref{eq:boundsNetwork} as well as
\begin{equation}\label{eq:NNvsPG}
\begin{aligned}
\|\mathbf{S}^\mathrm{nn}  - \mathbf{S}^\mathrm{pg}\|_2 &\leq \sum_{K \in \calT_H} \big\|\Phi_K\big(\mathbf{\Theta}_K - \mathbf{S}_{\Nb^\ell(K)}^\mathrm{pg}\big)\big\|_2 \\&
\leq \sum_{K \in \calT_H} \big\|\mathbf{\Theta}_K - \mathbf{S}_{\Nb^\ell(K)}^\mathrm{pg}\big\|_2
\lesssim H^{-d} H^k \lesssim H^{d + 1}.
\end{aligned}
\end{equation}
The difference of the C-LOD and the PG-LOD stiffness matrices can be bounded as follows. Let $\mathbf{v}$ be a vector (with corresponding function $v \in \VH$) such that
\begin{equation*}
\|\mathbf{S}^\mathrm{pg}  - \mathbf{S}^\mathrm{c}\|_2 = |\lambda_{\mathrm{max}}(\mathbf{S}^\mathrm{pg}  - \mathbf{S}^\mathrm{c})| = \frac{\big|\mathbf{v}^T (\mathbf{S}^\mathrm{pg}  - \mathbf{S}^\mathrm{c}) \mathbf{v}\big|}{\mathbf{v}^T\mathbf{v}}.
\end{equation*}
With the definition of $\mathbf{S}^\mathrm{pg}$ and $ \mathbf{S}^\mathrm{c}$ corresponding to the discrete problems~\eqref{eq:discvarproblempg} and \eqref{eq:discvarproblemclassic}, respectively, we thus have
\begin{align*}
\|\mathbf{S}^\mathrm{pg}  - \mathbf{S}^\mathrm{c}\|_2 & = \frac{|a((\mathsf{id} - \calQ^\ell)v,v) - a((\mathsf{id} - \calQ^\ell)v,(\mathsf{id} - \calQ^\ell)v)|}{\mathbf{v}^T\mathbf{v}} \\&
= \frac{|a((\mathsf{id} - \calQ^\ell)v,\calQ^\ell v)|}{\mathbf{v}^T\mathbf{v}} 
\lesssim \frac{|a((\calQ - \calQ^\ell)v,\calQ^\ell v)|}{H^{-d}\|v\|_{L^2(D)}^2}
\end{align*}
using the definition of the globally defined correction~$\calQ := \calQ^\infty$ and~\eqref{eq:normEquiv}  in the last step. Employing the estimate~\eqref{eq:loc}, the stability of the correction operator $\calQ^\ell$, and a classical inverse inequality, we arrive at
\begin{equation}\label{eq:PGvsCG}
\begin{aligned}
\|\mathbf{S}^\mathrm{pg}  - \mathbf{S}^\mathrm{c}\|_2 & \lesssim \frac{|a((\calQ - \calQ^\ell)v,\calQ^\ell v)|}{H^{-d}\|v\|_{L^2(D)}^2} \lesssim \frac{\beta\,\|\nabla(\calQ - \calQ^\ell)v\|_{L^2(D)}\,\|\nabla\calQ^\ell v\|_{L^2(D)}}{H^{-d}\|v\|_{L^2(D)}^2} \\&
\lesssim \frac{e^{-c_\mathrm{loc} \ell}\,\|\nabla v\|_{L^2(D)}^2}{H^{-d+2}\|\nabla v\|_{L^2(D)}^2} 
\lesssim H^{d-2} e^{-c_\mathrm{loc} \ell}.
\end{aligned}
\end{equation}
For the last term on the right-hand side of~\eqref{eq:proof2}, we estimate
\begin{align*}
\|\mathbf{u}^\mathrm{nn}\|_2 &\leq \|\mathbf{u}^\mathrm{pg} - \mathbf{u}^\mathrm{nn}\|_2 + \|\mathbf{u}^\mathrm{pg}\|_2 \leq \|\mathbf{u}^\mathrm{pg} - \mathbf{u}^\mathrm{nn}\|_2 + CH^{-d/2}\|\uPG\|_{L^2(D)} \\&\leq \|\mathbf{u}^\mathrm{pg} - \mathbf{u}^\mathrm{nn}\|_2 + CH^{-d/2}\|f\|_{L^2(D)},
\end{align*}
where we employ the stability of the solution $\uPG$ and the Friedrichs inequality. 
We now go back to~\eqref{eq:proof1} and~\eqref{eq:proof2} and obtain
\begin{equation}\label{eq:PGvsNN}
\begin{aligned}
\| \mathbf{u}^\mathrm{pg} - \mathbf{u}^\mathrm{nn}\|_2 &\leq CH^{-d/2}e^{-c_\mathrm{loc}\ell} + \big\|\big(\mathbf{S}^\mathrm{c}\big)^{-1}\big\|_2\, \big\|\mathbf{S}^\mathrm{nn}  - \mathbf{S}^\mathrm{c}\big\|_2\, \big\|\mathbf{u}^\mathrm{nn}\big\|_2 \\
& \leq  C H^{-d/2}e^{-c_\mathrm{loc}\ell} + \big\| \big(\mathbf{S}^\mathrm{c}\big)^{-1}\big\|_2\, \big\|\mathbf{S}^\mathrm{nn}  - \mathbf{S}^\mathrm{c}\big\|_2\, CH^{-d/2}\|f\|_{L^2(D)} \\&\qquad\qquad\qquad\qquad+ \big\| \big(\mathbf{S}^\mathrm{c}\big)^{-1}\big\|_2\, \big\|\mathbf{S}^\mathrm{nn}  - \mathbf{S}^\mathrm{c}\big\|_2\, \big\|\mathbf{u}^\mathrm{pg} - \mathbf{u}^\mathrm{nn}\big\|_2\\
& \leq C H^{-d/2}e^{-c_\mathrm{loc}\ell} +CH^{1-d/2}\|f\|_{L^2(D)} + \frac12 \|\mathbf{u}^\mathrm{pg} - \mathbf{u}^\mathrm{nn}\|_2.
\end{aligned}
\end{equation}
In the last step, we have used 
that
\begin{equation*}
\big\|\big(\mathbf{S}^\mathrm{c}\big)^{-1}\big\|_2	\big\|\mathbf{S}^\mathrm{nn}  - \mathbf{S}^\mathrm{c}\big\|_2 \lesssim H^{-d} \big( H^{d+1} + H^{d-2} e^{-c_\mathrm{loc} \ell}\big) \lesssim H \leq  \frac{1}{2}
\end{equation*}
according to~\eqref{eq:invMatrix}, \eqref{eq:NNvsPG}, and \eqref{eq:PGvsCG} if $H$ is chosen small enough and $\ell \gtrsim |\log (H)|$ large enough with respect to the respective hidden constant. 
Absorbing the last term on the right-hand side of~\eqref{eq:PGvsNN} leads to
\begin{align*}
\| \mathbf{u}^\mathrm{pg} - \mathbf{u}^\mathrm{nn}\|_2 &\lesssim H^{-d/2}e^{-c_\mathrm{loc}\ell} + H^{1-d/2}\,\|f\|_{L^2(D)}.
\end{align*}
Omitting the dependence on $f$ and using again that $\ell \gtrsim |\log (H)|$, we finally get
\[
\| \mathbf{u}^\mathrm{pg} - \mathbf{u}^\mathrm{nn}\|_2 \lesssim H^{1-d/2}.
\]
Employing~\eqref{eq:equiv}, this directly leads to
\[
\|\uPG - \uNN\|_{L^2(D)} \lesssim H. \qedhere
\]

\end{document}